\documentclass[a4paper, reqno]{amsart}
\usepackage[utf8]{inputenc}
\usepackage[T1]{fontenc}
\usepackage{lmodern}
\usepackage[all]{xy}
\usepackage{nicefrac,mathtools,enumitem}
\usepackage{microtype}
\usepackage{ifthen}
\usepackage{leftidx}
\usepackage{amssymb}
\usepackage{amsxtra}
\usepackage{xspace}
\usepackage{tikz}
\usetikzlibrary{matrix}
\tikzset{cd/.style=matrix of math nodes,row sep=2em,column sep=2em, text height=1.5ex, text depth=0.5ex}
\tikzset{cdar/.style=->,auto}

\setlist[enumerate,1]{label=\textup{(\arabic*)}}
\setlist[enumerate,2]{label=\textup{(\alph*)}}

\usepackage[pdftitle={The Drinfeld double for C*-algebraic quantum groups},
pdfauthor={Sutanu Roy},
pdfsubject={Mathematics}
]{hyperref}
\usepackage[lite]{amsrefs}
\newcommand*{\MRref}[2]{ \href{http://www.ams.org/mathscinet-getitem?mr=#1}{MR #1}}
\newcommand*{\arxiv}[1]{ \href{http://www.arxiv.org/abs/#1}{arXiv:#1}}
\renewcommand{\PrintDOI}[1]{\href{http://dx.doi.org/#1}{DOI #1}%
  \IfEmptyBibField{volume}{, (to appear in print)}{}}

\renewcommand{\PrintDOI}[1]{\href{http://dx.doi.org/\detokenize{#1}}{doi: \detokenize{#1}}%
  \IfEmptyBibField{pages}{, (to appear in print)}{}}

\BibSpec{book}{%
    +{}  {\PrintPrimary}                {transition}
    +{,} { \textit}                     {title}
    +{.} { }                            {part}
    +{:} { \textit}                     {subtitle}
    +{,} { \PrintEdition}               {edition}
    +{}  { \PrintEditorsB}              {editor}
    +{,} { \PrintTranslatorsC}          {translator}
    +{,} { \PrintContributions}         {contribution}
    +{,} { }                            {series}
    +{,} { \voltext}                    {volume}
    +{,} { }                            {publisher}
    +{,} { }                            {organization}
    +{,} { }                            {address}
    +{,} { \PrintDateB}                 {date}
    +{,} { }                            {status}
    +{}  { \parenthesize}               {language}
    +{}  { \PrintTranslation}           {translation}
    +{;} { \PrintReprint}               {reprint}
    +{.} { }                            {note}
    +{.} {}                             {transition}
    +{,} { \PrintDOI}                   {doi}
    +{}  {\SentenceSpace \PrintReviews} {review}
}

\numberwithin{equation}{section}

\theoremstyle{plain}
\newtheorem{theorem}[equation]{Theorem}
\newtheorem{lemma}[equation]{Lemma}
\newtheorem{proposition}[equation]{Proposition}
\newtheorem{corollary}[equation]{Corollary}

\theoremstyle{definition}
\newtheorem{definition}[equation]{Definition}
\newtheorem{notation}[equation]{Notation}

\theoremstyle{remark}
\newtheorem{remark}[equation]{Remark}

\newtheorem{example}[equation]{Example}

\newcommand*{\nb}{\nobreakdash}
\newcommand*{\Star}{\(^*\)\nb-}

\newcommand*{\C}{\mathbb C}
\newcommand*{\Z}{\mathbb Z}
\newcommand*{\R}{\mathbb R}

\newcommand*{\T}{\mathbb T}

\newcommand*{\Bialg}[1]{(#1,\Comult[#1])}
\newcommand*{\DuBialg}[1]{(\hat{#1},\DuComult[#1])}

\newcommand*{\G}[1][G]{\mathbb #1}
\newcommand*{\DuG}[1][G]{\widehat{\mathbb{#1}}}

\newcommand*{\Qgrp}[2]{\mathbb{#1}=\Bialg{#2}}
\newcommand*{\DuQgrp}[2]{\widehat{\mathbb{#1}}=\DuBialg{#2}}

\newcommand*{\Comult}[1][]{\Delta_{#1}}
\newcommand*{\DuComult}[1][]{\hat{\Delta}_{#1}}

\newcommand*{\Coinv}{\textup R}

\newcommand*{\Ker}{\mathrm{Ker}}
\newcommand*{\dom}{\mathrm{Dom}}
\newcommand*{\Bound}{\mathbb B}
\newcommand*{\Comp}{\mathbb K}

\newcommand*{\transpose}{\mathsf T}
\newcommand*{\Contvin}{\textup C_0}
\newcommand*{\Cont}{\textup C}
\newcommand*{\Mor}{\textup{Mor}}
\newcommand*{\Id}{\textup{id}}

\newcommand*{\Multunit}[1][]{\mathbb{W}^{#1}}
\newcommand*{\multunit}[1][]{\textup{W}^{#1}}
\newcommand*{\DuMultunit}[1][]{\widehat{\mathbb W}^{#1}}
\newcommand*{\Dumultunit}[1][]{\widehat{\textup W}^{#1}}

\newcommand*{\Bichar}{\mathbb{V}}
\newcommand*{\bichar}{\textup{V}}

\newcommand*{\Dubichar}{\hat{\bichar}}
\newcommand*{\DuBichar}{\widehat{\Bichar}}

\newcommand*{\HeisPair}[1]{(#1,\hat{#1})}
\newcommand*{\VHeisPair}{(\alpha,\beta)}

\newcommand*{\corep}[1]{\textup{#1}}                

\newcommand*{\Drinfdouble}[1]{\mathfrak{D}(#1)}
\newcommand*{\Codouble}[1]{\mathfrak{D}({#1}){ }\sphat\text{\space}}

\newcommand*{\GenDrinfdouble}[3]{\mathfrak{D}_{#3}(#1,#2)}
\newcommand*{\GenCodouble}[3]{\mathfrak{D}_{#3}(#1,#2){ }\sphat\text{\space}}

\newcommand*{\DrinfDoubAlg}{\mathcal{D}}
\newcommand*{\CodoubAlg}{\widehat{\mathcal{D}}}

\newcommand*{\Rmattxt}{\emph{R}}
\newcommand*{\Rmat}{\mathcal R}

\newcommand*{\Flip}{\Sigma}
\newcommand*{\flip}{\sigma}
\newcommand*{\Cst}{\textup C^*}
\newcommand*{\Cred}{\textup C^*_\textup r}

\newcommand*{\CLS}{\mathrm{CLS}}

\newcommand*{\Cstcat}{\mathfrak{C^*alg}}

\newcommand*{\Hils}[1][H]{\mathcal #1}

\newcommand*{\Mult}{\mathcal M}
\newcommand*{\U}{\mathcal U}

\newcommand*{\defeq}{\mathrel{\vcentcolon=}}

\newcommand*{\conj}[1]{\overline{#1}}

\begin{document}
\title{The Drinfeld double for \(\Cst\)\nb-algebraic quantum groups}

\author{Sutanu Roy}
\email{sr26@uOttawa.ca}
\address{Department of Mathematics and Statistics\\
              585 King Edward\\ 
              K1N 6N5 Ottawa\\
              Canada}

\begin{abstract}
 In this article, we establish the duality between the generalised 
 Drinfeld double and generalised quantum codouble within the framework 
 of modular or manageable (not necessarily regular) multiplicative unitaries, 
 and discuss several properties. 
 \end{abstract}

\maketitle

\section{Introduction}
  \label{sec:intro}      
  One of the milestones in the theory of Hopf algebras 
  is the \emph{quantum double} or \emph{Drinfeld double} construction  
  \cite{Drinfeld:Quantum_groups}. This has subsequently been generalised 
  in several algebraic frameworks~\cites{Majid:More_ex_bicross, Delvaux-vanDaele:Drinf_doub_mult_hopf}. 
  
  The quantum double construction for analytic quantum groups was developed 
  in many different frameworks, along with the development of 
  a general theory of compact and locally compact quantum groups.
  In fact, the terms ``quantum double'' and ``double crossed product'' in 
  the context of locally compact quantum groups actually refers to the 
  (analytically generalised) dual of the respective constructions in the algebraic
  framework. 

  In~\cites{Podles-Woronowicz:Quantum_deform_Lorentz},  
  Podle\'s and Woronowicz generalised the quantum double construction 
  for compact quantum groups~\cite{Woronowicz:Compact_pseudogroups}, under the 
  name \emph{double group} construction. They constructed \(q\)\nb-deformations of 
  \(\textup{SL}(2,\C)\) as the double group of~\(\textup{SU}_{q}(2)\) groups 
  \cite{Woronowicz:Twisted_SU2} for \(q\in [-1,1]\setminus\{0\}\). 
  
  The first step towards the general theory of (topological) locally compact quantum groups, 
  in the \(\Cst\)\nb-algebraic framework, goes back to the work of Baaj and Skandalis 
  \cite{Baaj-Skandalis:Unitaires}. As basic axioms they used a unitary operator,
  called the \emph{multiplicative unitary}, having the \emph{regularity} property. Also their 
  \emph{\(Z\)\nb-produit tensoriel} for regular multiplicative unitaries (see~\cite{Baaj-Skandalis:Unitaires}*{Section 8}) 
  generalises the quantum double construction for compact quantum groups. Unfortunately, multiplicative unitaries related to 
  locally compact quantum groups are not always regular (see~\cite{Baaj:Regular-Rep-Woronowicz-shift}).
  The notion \emph{manageability} of multiplicative unitaries, introduced by Woronowicz 
  in~\cite{Woronowicz:Multiplicative_Unitaries_to_Quantum_grp}, provides a more general 
  approach to the \(\Cst\)\nb-algebraic theory for locally compact quantum groups or, in short, 
  \emph{\(\Cst\)\nb-quantum groups} (see Definition~\ref{def:Qnt_grp}).

  A general theory of (measure theoretic) locally compact quantum groups  was proposed 
  by Kustermans and Vaes \cites{Kustermans-Vaes:LCQG, Kustermans-Vaes:LCQGvN} and Masuda, 
  Nakagami and Woronowicz~\cite{Masuda-Nakagami-Woronowicz:C_star_alg_qgrp}, 
  assuming existence of Haar weights. Also, \cite{Kustermans-Vaes:LCQG}*{Proposition 6.10} 
  shows that the left (respectively right) regular representation associated to the left (respectively right) 
  Haar weight is a manageable multiplicative unitary. According to 
  \cite{Baaj-Skandalis-Vaes:Non-semi-regular}*{Terminology 5.4}, 
  a locally compact quantum is regular if its regular representation is a regular multiplicative unitary. 
  
  In~\cite{Baaj-Vaes:Double_cros_prod}, Baaj and Vaes developed the general theory of the double crossed product of 
  a matched pair of locally compact quantum groups. Tomita\nb--Takesaki operators of the respective quantum groups 
  play the key role in their construction. A matching of two von Neumann algebraic quantum groups \(\Bialg{M_{1}}\) 
  and~\(\Bialg{M_{2}}\) is a normal faithful \Star{}homomorphism \(m\colon M_{1}\bar{\otimes} 
  M_{2}\to M_{1}\bar{\otimes} M_{2}\) with some additional property. Here \(\bar{\otimes}\) 
  denotes the von Neumann algebraic tensor product. The underlying von Neumann algebra~\(M_{m}\) 
  of the associated double crossed product is~\(M_{1}\bar{\otimes}M_{2}\). In particular, every bicharacter 
  \(\bichar\in M_{1}\bar{\otimes} M_{2}\) defines an inner matching defined by 
  \(m(x)= \bichar (x)\bichar^*\) for all~\(x\in M_{1}\bar{\otimes} M_{2}\). 
  The double crossed product associated to an inner matching, is called 
  \emph{generalised quantum double} (see~\cite{Baaj-Vaes:Double_cros_prod}*{Section 8}). The word 
  ``generalised'' refers to the generalisation of the quantum double construction for quasi Woronowicz algebras 
  \cite{Masuda-Nakagami:quasi-Woronowicz-alg} (previously regarded as locally compact quantum groups in the von 
  Neumann algebra framework) by Yamanouchi~\cite{Yamanouchi:Double_grp_const_vN}. 
  
  In~\cite{Woronowicz-Zakrzewski:Quantum_Lorentz_Gauss}, Woronowicz and Zakrzewski 
  constructed another \(q\)\nb-deformation (for \(q\in(0,1)\)) of \(\textup{SL}(2,\C)\), 
  in the \(\Cst\)\nb-algebraic framework, as the quantum double (under 
  the name double group) of the quantum \(E(2)\) group, which is 
  not regular (see~\cite{Baaj:Regular-Rep-Woronowicz-shift}). 
  Therefore, the \(\Cst\)\nb-algebraic description for~\(M_{m}\), 
  given by Baaj and Vaes \cite{Baaj-Vaes:Double_cros_prod}*{Proposition 9.5} 
  does not cover this example because it assumes regularity on both 
  the locally compact quantum groups \(\Bialg{M_{1}}\) and \(\Bialg{M_{2}}\). 
      
 In this article, we construct and establish the duality between 
 \emph{generalised Drinfeld doubles} and 
 \emph{generalised quantum codoubles} 
 (called as generalised quantum doubles in 
 \cite{Baaj-Vaes:Double_cros_prod}*{Section 8}) 
 in the general framework of manageable multiplicative unitaries. 
 Therefore, our work generalises the~\(\Cst\)\nb-algebraic picture of the 
 generalised quantum doubles in~\cite{Baaj-Vaes:Double_cros_prod}*{Section 9}, 
 as we do not need to assume neither Haar measures nor regularity on the factor quantum groups. 
 In particular, our work also generalises the \emph{quantum codouble} 
 (sometimes called double group \cites{Podles-Woronowicz:Quantum_deform_Lorentz, Woronowicz-Zakrzewski:Quantum_Lorentz_Gauss}, 
 quantum double~\cite{Yamanouchi:Double_grp_const_vN}, and Drinfeld 
 double~\cite{Nest-Voigt:Poincare}*{Section 3}) construction for locally compact quantum groups with 
 Haar weights \cite{Masuda-Nakagami-Woronowicz:C_star_alg_qgrp}*{Section 8}, and  
 for manageable multiplicative unitaries (an unpublished work of 
 S.L. Woronowicz presented at RIMS in 2011). 
   
Let us briefly outline the structure of this article. In Section~\ref{sec:prelim}, 
we recall basic necessary preliminaries. In particular, the main results 
on modular and manageable multiplicative unitaries, that give rise to 
\(\Cst\)\nb-quantum groups~\cite{Woronowicz:Multiplicative_Unitaries_to_Quantum_grp}, 
coactions and corepresentations of \(\Cst\)\nb-quantum groups, and several equivalent notions 
of homomorphisms of \(\Cst\)\nb-quantum groups~\cite{Meyer-Roy-Woronowicz:Homomorphisms} 
are stated.

Let~\(\Qgrp{G}{A}\) and~\(\Qgrp{H}{B}\) be \(\Cst\)\nb-quantum groups 
(in the sense of Definition~\ref{def:Qnt_grp}), and \(\bichar\in\U(\hat{A}\otimes 
\hat{B})\) be a bicharacter (in the sense of Definition~\ref{def:bicharacter}).

In Section~\ref{sec:Heisenberg}, we recall the concept of \(\bichar\)\nb-Heisenberg pairs 
from~\cite{Meyer-Roy-Woronowicz:Twisted_tensor}. Then we introduce the 
notion of \(\bichar\)\nb-\emph{Drinfeld pair} in Section~\ref{sec:Drinf}, which plays the 
fundamental role in this article. Roughly, it is a pair of representations 
\((\rho,\theta)\) of~\(A\) and~\(B\) on some Hilbert space \(\Hils\) 
satisfying certain commutation relations governed by~\(\bichar\). 
We systematically construct a \(\bichar\)\nb-Drinfeld pair 
and a modular multiplicative unitary, denoted by \(\Multunit[\DrinfDoubAlg]\) 
(see Theorem~\ref{the:multunit_drinf}). Section~\ref{sec:gen_qnt_codoub} is 
devoted to the construction of the generalised Drinfeld double, as a \(\Cst\)\nb-quantum 
group from~\(\Multunit[\DrinfDoubAlg]\), and generalised quantum codouble as its 
dual. 

In particular, the generalised Drinfeld double construction for a trivial bicharacter 
yields the usual product of the respective \(\Cst\)\nb-quantum 
groups (see~Example~\ref{ex:triv_bichar}). In Section~\ref{sec:prop_Gen_Drinf}, 
we extend certain known results for the product of groups and 
\(\Cst\)\nb-quantum groups to generalised Drinfeld 
doubles. It is well known that the Drinfeld double of a finite dimensionsal 
Hopf algebra has an~\(\Rmattxt\)\nb-matrix~\cite{Drinfeld:Quantum_groups}. 
This was generalised in several analytic contexts~\cites{Baaj-Skandalis:Unitaires, 
Yamanouchi:Double_grp_const_vN, Delvaux-vanDaele:Drinf_doub_mult_hopf}. 
We extend this result in the context of modular or manageable multiplicative 
unitaries. Finally, in Section~\ref{sub:Prop_Gen_qnt_codoub}, we 
discuss the coaction and corepresentation of generalised quantum 
codoubles.
\section{Preliminaries} 
  \label{sec:prelim}     
   Throughout we use the symbol ``:='' to abbreviate the phrase ``defined by''.
   
     All Hilbert spaces and \(\Cst\)\nb-algebras are 
     assumed to be separable. 
      
    For two norm\nb-closed subsets~\(X\) and~\(Y\) of a~\(\Cst\)\nb-algebra, 
    let 
    \[
       X\cdot Y\defeq\{xy : x\in X, y\in Y\}^{\textup{CLS}},
    \]
  where CLS stands for the~\emph{closed linear span}. 
   
  For a~\(\Cst\)\nb-algebra~\(A\), let~\(\Mult(A)\) be its multiplier algebra and 
  \(\U(A)\) be the group of unitary multipliers of~\(A\). The unit of 
  \(\Mult(A)\) is denoted by~\(1_{A}\). 
   Next recall some standard facts about multipliers and morphisms of 
   \(\Cst\)\nb-algebras from~\cite{Masuda-Nakagami-Woronowicz:C_star_alg_qgrp}*{Appendix A}.
   Let~\(A\) and~\(B\) be~\(\Cst\)\nb-algebras. A \Star{}homomorphism 
  \(\varphi\colon A\to\Mult(B)\) is called \emph{nondegenerate} if 
  \(\varphi(A)\cdot B=B\). Each nondegenerate \Star{}homomorphism 
  \(\varphi\colon A\to\Mult(B)\) extends uniquely to a unital 
  \Star{}homomorphism \(\widetilde{\varphi}\) from~\(\Mult(A)\) to 
  \(\Mult(B)\). Let \(\Cstcat\) be the category of \(\Cst\)\nb-algebras with nondegenerate 
  \Star{}homomorphisms \(A\to\Mult(B)\) as morphisms \(A\to B\); let Mor(A,B) denote this 
  set of morphisms. We use the same symbol for an element of~\(\Mor(A,B)\) and its 
  unique extenstion from~\(\Mult(A)\) to~\(\Mult(B)\).
  
  Let~\(\conj{\Hils}\) be the conjugate Hilbert space to the Hilbert
  space~\(\Hils\). The \emph{transpose} of an operator
  \(x\in\Bound(\Hils)\) is the operator
  \(x^\transpose\in\Bound(\conj{\Hils})\) defined by
 \(x^\transpose(\conj{\xi}) \defeq \conj{x^*\xi}\) for all
 \(\xi\in\Hils\).  The transposition is a linear, involutive
 anti-automorphism \(\Bound(\Hils)\to\Bound(\conj{\Hils})\).  

  A \emph{representation} of a
  \(\Cst\)\nb-algebra~\(A\) on a Hilbert space~\(\Hils\) is a 
  nondegenerate \Star{}homomorphism \(\pi\colon A\to\Bound(\Hils)\). 
  Since \(\Bound(\Hils)=\Mult(\Comp(\Hils))\), the nondegeneracy
  conditions \(\pi(A)\cdot\Comp(\Hils)=\Comp(\Hils)\) is equivalent 
  to begin \(\pi(A)(\Hils)\) is norm dense in~\(\Hils\),  
  and hence this is same as having a morphism 
  from~\(A\) to~\(\Comp(\Hils)\). The identity representation 
  of~\(\Comp(\Hils)\) on \(\Hils\) is denoted by~\(\Id_{\Hils}\).
  The group of unitary operators on a Hilbert space~\(\Hils\) 
  is denoted by \(\U(\Hils)\). The identity element in 
  \(\U(\Hils)\) is denoted by~\(1_{\Hils}\).

  We use~\(\otimes\) both for the tensor product of Hilbert spaces  
  and minimal tensor product of \(\Cst\)\nb-algebras, which is 
  well understood from the context.
  We write~\(\Flip\) for the tensor flip \(\Hils\otimes\Hils[K]\to
  \Hils[K]\otimes\Hils\), \(x\otimes y\mapsto y\otimes x\), for two 
  Hilbert spaces \(\Hils\) and~\(\Hils[K]\).  We write~\(\flip\) for the
  tensor flip isomorphism \(A\otimes B\to B\otimes A\) for two
  \(\Cst\)\nb-algebras \(A\) and~\(B\). 
  
 Let~\(A_{1}\), \(A_{2}\), \(A_{3}\) be \(\Cst\)\nb-algebras. 
 For any~\(t\in\Mult(A_{1}\otimes A_{2})\) we denote the 
 leg numberings on the level of~\(\Cst\)\nb-algebras as 
 \(t_{12}\defeq t\otimes 1_{A_{3}} \in\Mult(A_{1}\otimes A_{2}\otimes A_{3})\), 
 \(t_{23}\defeq1_{A_{3}}\otimes t_{12}\in\Mult(A_{3}\otimes A_{1}\otimes A_{2})\)
 and~\(t_{13}\defeq\flip_{12}(t_{23})=\flip_{23}(t_{12})\in\Mult(A_{1}\otimes A_{3}\otimes A_{2})\). 
 In particular,  let \(A_{i}=\Bound(\Hils_{i})\) for some Hilbert spaces~\(\Hils_{i}\), where 
 \(i=1,2,3\). Then for any \(t\in\Bound(\Hils_{1}\otimes\Hils_{2})\) 
 the leg numberings are obtained by replacing~\(\flip\) with the 
 conjugation by~\(\Flip\) operator.

 \subsection{Multiplicative unitaries and quantum groups}
  \label{sec:multunit_quantum_groups}
   \begin{definition}[\cite{Baaj-Skandalis:Unitaires}*{D\'{e}finition 1.1}]
    \label{def:multunit}
   Let~\(\Hils\) be a Hilbert space.  A unitary
  \(\Multunit\in\U(\Hils\otimes\Hils)\) is \emph{multiplicative} if it
   satisfies the \emph{pentagon equation}
   \begin{equation}
    \label{eq:pentagon}
     \Multunit_{23}\Multunit_{12}
     = \Multunit_{12}\Multunit_{13}\Multunit_{23}
     \qquad
     \text{in \(\U(\Hils\otimes\Hils\otimes\Hils).\)}
   \end{equation}
 \end{definition}
 Technical assumptions such as manageability
 (\cite{Woronowicz:Multiplicative_Unitaries_to_Quantum_grp}) or, more
 generally, modularity (\cite{Soltan-Woronowicz:Remark_manageable}) are
 needed in order to construct a \(\Cst\)\nb-algebras out of a multiplicative unitary.

 \begin{definition}[\cite{Soltan-Woronowicz:Remark_manageable}*{Definition 2.1}]
   \label{def:modularity}
   A multiplicative unitary~\(\Multunit\in\U(\Hils\otimes\Hils)\) is 
   \emph{modular} if there are positive self-adjoint operators~\(Q\) and~\(\hat{Q}\)
    acting on~\(\Hils\) and~\(\widetilde{\Multunit}\in\U(\conj{\Hils}\otimes\Hils)\) such that: 
    \begin{enumerate}[label=(\roman*)]
     \item\label{eq:modular_Q_hat_Q} 
           \(\Ker(Q)=\Ker(\hat{Q})=\{0\}\) 
           and~\(\Multunit(\hat{Q}\otimes Q)\Multunit[*]=(\hat{Q}\otimes Q)\),
     \item\label{eq:modular_cond} 
           \(\big(x\otimes u\mid\Multunit\mid z\otimes y\big)
          =\big(\conj{z}\otimes Qu\mid\widetilde{\Multunit}\mid\conj{x}\otimes Q^{-1}y\big)\)
           for all~\(x,z\in\Hils\), \(u\in\dom(Q)\) and~\(y\in\dom(Q^{-1})\),
    \end{enumerate} 
    where~\(\conj{\Hils}\) is the complex-conjugate Hilbert space associated to~\(\Hils\).
        
    If~\(\hat{Q}=Q\) then~\(\Multunit\) is called~\emph{manageable}.
   \end{definition} 

\begin{theorem}[\cites{Soltan-Woronowicz:Remark_manageable,
    Soltan-Woronowicz:Multiplicative_unitaries}]
  \label{the:Cst_quantum_grp_and_mult_unit}
  Let~\(\Hils\) be a Hilbert space and
  \(\Multunit\in\U(\Hils\otimes\Hils)\) a modular
  multiplicative unitary.  Let
  \begin{alignat}{2}
    \label{eq:first_leg_slice}
    A &\defeq \{(\omega\otimes\Id_{\Hils})\Multunit :
    \omega\in\Bound(\Hils)_*\}^\CLS,\\
    \label{eq:second_leg_slice}
    \hat{A} &\defeq \{(\Id_{\Hils}\otimes\omega)\Multunit :
    \omega\in\Bound(\Hils)_*\}^\CLS.
  \end{alignat}
  \begin{enumerate}
  \item\label{eq:A_and_A-hat} \(A\) and \(\hat{A}\) are separable, nondegenerate
    \(\Cst\)\nb-subalgebras of~\(\Bound(\Hils)\).
  \item\label{eq:Red_bichar} \(\Multunit\in\U(\hat{A}\otimes
    A)\subseteq\U(\Hils\otimes\Hils)\).  We write~\(\multunit[A]\)
    for~\(\Multunit\) viewed as a unitary multiplier of
    \(\hat{A}\otimes A\). 
  \item\label{eq:Comul} There is a unique \(\Comult[A]\in\Mor(A,A\otimes A)\) such
    that
    \begin{equation}
      \label{eq:W_char_in_second_leg}
      (\Id_{\hat{A}}\otimes \Comult[A])\multunit[A]
      = \multunit[A]_{12}\multunit[A]_{13}
      \qquad \text{in \(\U(\hat{A}\otimes A\otimes A)\);}
    \end{equation}
    it is \emph{coassociative}:
    \begin{equation}
      \label{eq:coassociative}
      (\Comult[A]\otimes\Id_A)\Comult[A]
      = (\Id_A\otimes\Comult[A])\Comult[A],
    \end{equation}
    and satisfies the \emph{cancellation laws}
    \begin{equation}
      \label{eq:Podles}
      \Comult[A](A)\cdot(1_A\otimes A)
      = A\otimes A
      = (A\otimes 1_A)\cdot\Comult[A](A).
    \end{equation}
      \item\label{eq:Coinv} There is a unique closed linear operator~\(\kappa_{A}\) on the Banach space~\(A\) such that
            \(\{(\omega\otimes\Id_{A})\multunit :\omega\in\hat{A}'\}\) is a core for~\(\kappa_{A}\) 
            and
            \[
              \kappa_{A}((\omega\otimes\Id_{A})\multunit)=(\omega\otimes\Id_{A})\multunit[*]
            \]
            for any~\(\omega\in\hat{A}'\). Moreover, for all~\(a,b\in\dom(\kappa_{A})\) the product 
            \(ab\in\dom(\kappa_{A})\) and~\(\kappa_{A}(ab)=\kappa_{A}(b)\kappa_{A}(a)\),  
            the image~\(\kappa_{A}(\dom(\kappa_{A}))\) coincides with~\(\dom(\kappa_{A})^{*}\),
            and~\(\kappa_{A}(\kappa_{A}(a)^{*})^{*}=a\) for all~\(a\in\dom(\kappa_{A})\).
      \item\label{eq:sc_grp} there is a unique one\nb-parameter group, called the ~\emph{scaling group}),
               \(\{\tau^{A}_{t}\}_{t\in\R}\) of \Star{}automorphisms of \(A\) and a unique ultraweakly 
               continuous, involutive, \Star{}anti-automorphism, called the~\emph{unitary antipode}, 
               \(\Coinv_A\) 
               of~\(A\) such that 
               \begin{enumerate}[label=\textup{(\roman*)}]
                 \item\label{eq:pol_dec_coinv} \(\kappa_{A}=\Coinv_{A}\tau^{A}_{\mathrm{i}/2}\);
                 \item\label{eq:bdd_coinv_sc_grp} \(\Coinv_{A}\) commutes with~\(\tau^{A}_{t}\) for all~\(t\in\R\) and 
                          \(\dom(\kappa_{A})=\dom(\tau^{A}_{\mathrm{i}/2})\),
                 \item\label{eq:comul_sc_grp}  \(\Comult[A]\tau^{A}_{t}=(\tau^{A}_{t}\otimes\tau^{A}_{t})\Comult[A]\) 
                       for all~\(t\in\R\),
                 \item\label{eq:bdd_coinv_comul} \(\Comult[A] \Coinv_{A} =\flip(\Coinv_{A}\otimes \Coinv_{A})\Comult[A]\), 
                        where~\(\flip\) denotes the flip map.
           \end{enumerate}
    \item\label{eq:mang_rels} Let~\(Q\) and~\(\widetilde{\Multunit}\) be the operators associated to~\(\Multunit\) in
         Definition~\textup{\ref{def:modularity}}. Then,
         \begin{enumerate}[label=\textup{(\roman*)}]
          \item\label{eq:sc_grp_Q} 
          for any~\(t\in\R\) and~\(a\in A\) we have~\(\tau^{A}_{t}(a)=Q^{2\mathrm{i}t}aQ^{-2\mathrm{i}t}\),
          \item\label{eq:conj_mult} writing~\(a^{\Coinv_{A}}\) instead of~\(\Coinv_{A}(a)\), we have
                      \(\multunit[\transpose\otimes\Coinv_{A}]=\widetilde{\Multunit}^{*}\), where the left hand side 
                      is viewed as a unitary on~\(\conj{\Hils}\otimes\Hils\).
         \end{enumerate}
  \end{enumerate}
\end{theorem}
In general, a pair~\(\Bialg{A}\) consisting of a~\(\Cst\)\nb-algebra~\(A\) and 
a morphism~\(\Comult[A]\in\Mor(A,A\otimes A)\) satisfying coassociativity 
condition~\eqref{eq:coassociative} and~\eqref{eq:Podles} is called 
a \emph{bisimplifiable \(\Cst\)\nb-bialgebra} (see~\cite{Baaj-Skandalis:Unitaires}*{Definition 0.1}). 
Two such pairs~\(\Bialg{A}\) and~\(\Bialg{B}\) are isomorphic if there is an isomorphism 
\(\varphi\in\Mor(A,B)\) intertwining the comultiplications: \((\varphi\otimes\varphi)\Comult[A]
=\Comult[B]\varphi\).  
\begin{definition}[\cite{Soltan-Woronowicz:Multiplicative_unitaries}*{Definition 3}]
  \label{def:Qnt_grp}
   Let A be a~\(\Cst\)\nb-algebra and \(\Comult[A]\in\Mor(A,A\otimes A)\). 
    Then the pair~\(\Qgrp{G}{A}\) is a \emph{\(\Cst\)\nb-quantum group} if there is a modular
    multiplicative unitary~\(\Multunit\in\U(\Hils\otimes\Hils)\) such that~\(\Bialg{A}\)
     is isomorphic to the~\(\Cst\)\nb-algebra with comultiplication associated to 
     \(\Multunit\) as described in Theorem~\ref{the:Cst_quantum_grp_and_mult_unit}. 
     Then we say~\(\Qgrp{G}{A}\) is \emph{generated} by~\(\Multunit\).
\end{definition}

 The notions of modularity and manageability are not very far from each other:
 starting from a modular multiplicative unitary one can construct a manageable 
 multiplicative unitary on a different Hilbert space (see~\cite{Soltan-Woronowicz:Remark_manageable}) 
 giving rise to the same~\(\Cst\)\nb-quantum group. Therefore, we shall consider 
 only manageable multiplicative unitaries from now on.

The \emph{dual} multiplicative unitary is \(\DuMultunit\defeq
\Flip\Multunit[*]\Flip\in\U(\Hils\otimes\Hils)\). 
It is modular or manageable if~\(\Multunit\) is.  The \(\Cst\)\nb-quantum group
\(\DuQgrp{G}{A}\) generated by~\(\DuMultunit\) is the \emph{dual}
of~\(\G\). Define~\(\Dumultunit[A]\in\U(A\otimes\hat{A})\) by 
\(\Dumultunit[A]\defeq\flip((\multunit[A])^{*})\in\U(A\otimes\hat{A})\), where 
\(\flip(\hat{a}\otimes a)=a\otimes\hat{a}\). It satisfies
\begin{equation}
 \label{eq:aux_W_char_in_first_leg}
    (\Id_{A}\otimes\DuComult[A])\Dumultunit[A]=\Dumultunit[A]_{12}\Dumultunit[A]_{13}
  \qquad\text{in~\(\U(A\otimes\hat{A}\otimes\hat{A})\).}
\end{equation}
Equivalently, we get the character condition on the first leg 
of~\(\multunit[A]\):
\begin{equation}
  \label{eq:W_char_in_first_leg}
  (\DuComult[A]\otimes\Id_{A})\multunit[A]=\multunit[A]_{23}\multunit[A]_{13}
  \qquad\text{in~\(\U(\hat{A}\otimes\hat{A}\otimes A)\).}
\end{equation}   
\begin{definition}[\cite{Soltan-Woronowicz:Multiplicative_unitaries}*{page 53}]
 \label{def:red_bichar}
 The unitary \(\multunit[A]\in\Mult(\hat{A}\otimes A)\) is called the \emph{reduced bicharacter} 
 for~\((\G,\DuG)\). Equivalently, \(\Dumultunit[A]\in\U(A\otimes\hat{A})\) is the reduced bicharacter 
 for~\((\DuG,\G)\).  
\end{definition}
\begin{theorem}[\cite{Soltan-Woronowicz:Multiplicative_unitaries}*{Theorem 5}]
 \label{the:ind_multunit}
 The \(\Cst\)\nb-quantum group \(\Qgrp{G}{A}\) is independent of the choice of the modular 
 multiplicative unitary that generates~\(\G\). Furthermore, the dual \(\Cst\)\nb-quantum group 
 \(\DuQgrp{G}{A}\) and the reduced bicharacter~\(\multunit[A]\in\U(\hat{A}\otimes A)\) are 
  determined uniquely (up to isomorphism) by~\(\G\).
\end{theorem} 

\begin{definition}
  \label{def:corepresentation}
  A (unitary) \emph{corepresentation} of~\(\G\) on a \(\Cst\)\nb-algebra 
  \(C\) is an element \(\corep{U}\in\U(C\otimes A)\)
  with
  \begin{equation}
    \label{eq:corep_cond}
    (\Id_{C}\otimes\Comult[A])\corep{U} =\corep{U}_{12}\corep{U}_{13}
    \qquad\text{in }\U(C\otimes A\otimes A).
  \end{equation}
  In particular, \(\corep{U}\) is said to be a corepresentation of~\(\G\) on a 
  Hilbert space~\(\Hils\) whenever \(C=\Comp(\Hils)\).
\end{definition}
\begin{example}
  \label{ex:unitary_representations}
  The \emph{trivial} corepresentation of \(\G\) on a Hilbert space~\(\Hils\) is 
  \(\corep{U}=1_{\Hils}\otimes1_{A}\in\U(\Comp(\Hils)\otimes A)\). Equation~\ref{eq:W_char_in_second_leg} 
  shows that the reduced bicharacter~\(\multunit[A]\in\U(\hat{A}\otimes A)\) is 
  a corepresentation of~\(\G\) on~\(\hat{A}\).
\end{example}

\begin{definition}
  \label{def:cont_coaction}
  A \emph{\textup(right\textup) coaction} of~\(\G\)  or \emph{\(\G\)\nb-coaction} 
 on a \(\Cst\)\nb-algebra~\(C\) is a morphism \(\gamma\colon C\to C\otimes
  A\) with the following properties:
  \begin{enumerate}
  \item \(\gamma\) is injective;
  \item \(\gamma\) is a comodule structure, that is,
    \begin{equation}
      \label{eq:right_coaction}
       (\Id_{C}\otimes\Comult[A])\gamma=(\gamma\otimes\Id_{A})\gamma ;
    \end{equation}
  \item \(\gamma\) satisfies the \emph{Podleś condition}:
  \begin{equation}
   \label{eq:Podles_cond}
    \gamma(C)\cdot(1_C\otimes A)=C\otimes A .
   \end{equation}
  \end{enumerate}
 \end{definition}  
 
\begin{example}
  \label{ex:continuous_coactions}
  The \emph{trivial} coaction of \(\G\) on a \(\Cst\)\nb-algebra~\(C\), is 
  defined by \(\tau\colon C\to C\otimes A\), \(c\mapsto c\otimes 1_A\). 
  The cancellation law~\eqref{eq:Podles} implies that 
  \(\Comult[A]\colon A\to A\otimes A\) is a \(\G\)\nb-coaction 
  on~\(A\).  More generally, \(\Id_C\otimes\Comult[A]\colon C\otimes A\to C
  \otimes A\otimes A\) is a \(\G\)\nb-coaction on \(C\otimes A\) for any
  \(\Cst\)\nb-algebra~\(C\).  Lemma~\(2.9\) in~\cite{Meyer-Roy-Woronowicz:
  Twisted_tensor} says that any coaction may be embedded into one of 
  this form.
\end{example}
  A pair \((C,\gamma)\) consisting of a~\(\Cst\)\nb-algebra \(C\) 
  and a \(\G\)\nb-coaction~\(\gamma\) on \(C\) 
  is called a \emph{\(\G\)\nb-\(\Cst\)\nb-algebra} . 
  A morphism \(f\colon C\to D\) between two
  \(\G\)\nb-\(\Cst\)\nb-algebras \((C,\gamma)\) and \((D,\delta)\) is
  \emph{\(\G\)\nb-\hspace{0pt}equivariant} if \(\delta f =
  (f\otimes\Id_A)\gamma\). Let \(\Cstcat(\G)\) denote the category 
  with \(\G\)\nb-\(\Cst\)-algebras as objects and \(\G\)\nb-equivariant
 morphisms as arrows. 
\subsection{Quantum group homomorphisms}
\label{sec:bicharacters_morphisms}
Let \(\Qgrp{G}{A}\) and \(\Qgrp{H}{B}\) be \(\Cst\)\nb-quantum groups.
Let \(\DuQgrp{G}{A}\) and \(\DuQgrp{H}{B}\) be their duals.
\begin{definition}[\cite{Meyer-Roy-Woronowicz:Homomorphisms}*{Definition 16}]
  \label{def:bicharacter}
  A \emph{bicharacter from \(\G\) to~\(\DuG[H]\)} is a unitary
  \(\bichar\in\U(\hat{A}\otimes \hat{B})\) with
  \begin{alignat}{2}
    \label{eq:bichar_char_in_first_leg}
    (\DuComult[A]\otimes\Id_{\hat{B}})\bichar
    &=\bichar_{23}\bichar_{13}
    &\qquad &\text{in }
    \U(\hat{A}\otimes\hat{A}\otimes \hat{B}),\\
    \label{eq:bichar_char_in_second_leg}
    (\Id_{\hat{A}}\otimes\DuComult[B])\bichar
    &=\bichar_{12}\bichar_{13}
    &\qquad &\text{in }
    \U(\hat{A}\otimes \hat{B}\otimes \hat{B}).
  \end{alignat}
\end{definition}
  A \emph{Hopf~\(^*\)\nb-homomorphism} from~\(\G\) to 
  \(\DuG[H]\) is an element~\(f\in\Mor(A,\hat{B})\) 
  that intertwines the comultiplications: 
  \begin{equation} 
    \label{eq:hopf_star_hom}
      (f\otimes f)\Comult[A](a)= \DuComult[B] f(a) 
      \quad\text{for~\(a\in A\).}
 \end{equation}      
  Then~\(\bichar_{f}\defeq(\Id_{\hat{A}}\otimes f)\multunit[A]\in\U(\hat{A}\otimes 
  \hat{B})\) is a bicharacter from~\(\G\) to~\(\G[H]\). We say that~\(\bichar_{f}\) is 
  \emph{induced} by~\(f\). 

Bicharacters in \(\U(\hat{A}\otimes B)\) are interpreted as quantum
group morphisms from~\(\G\) to~\(\G[H]\)
in~\cite{Meyer-Roy-Woronowicz:Homomorphisms}.  We shall use
bicharacters in \(\U(\hat{A}\otimes \hat{B})\) throughout. Let us recall 
some definitions from~\cite{Meyer-Roy-Woronowicz:Homomorphisms} in this
setting.

\begin{definition}
  \label{def:right_quantum_morphism}
  A \emph{right quantum group homomorphism} from~\(\G\) to~\(\DuG[H]\) is
  a morphism \(\Delta_R\colon A\to A\otimes\hat{B}\) with the following 
  properties:
  \begin{equation}
    \label{eq:right_homomorphism}
     (\Comult[A]\otimes\Id_{\hat{B}})\Delta_{R} 
   =(\Id_{A}\otimes\Delta_{R})\Comult[A]
     \quad\text{and}\quad 
     (\Id_{A}\otimes\DuComult[B])\Delta_{R}
  =(\Delta_{R}\otimes\Id_{\hat{B}})\Delta_{R}.
 \end{equation}           
 Similarly, a \emph{left quantum group homomorphism} from~\(\G\) to~\(\DuG[H]\) is a 
 morphism \(\Delta_L\colon A\to\hat{B}\otimes A\) satisfying the following properties:
 \begin{equation}
  \label{eq:left_homomorphism}
   (\Id_{\hat{B}}\otimes\Comult[A])\Delta_{L}
= (\Delta_{L}\otimes\Id_{A})\Comult[A]
   \quad\text{and}\quad
   (\DuComult[B]\otimes\Id_{A})\Delta_{L}
=(\Id_{\hat{B}}\otimes\Delta_{L})\Delta_{L}.     
 \end{equation} 
 \end{definition}
The following theorem summarises some of the main results
of~\cite{Meyer-Roy-Woronowicz:Homomorphisms}.
\begin{theorem}
  \label{the:equivalent_notion_of_homomorphisms}
  There are natural bijections between the following sets:
  \begin{enumerate}
  \item bicharacters \(\bichar\in\U(\hat{A}\otimes\hat{B})\)
    from~\(\G\) to~\(\DuG[H]\);
  \item bicharacters \(\Dubichar\in\U(\hat{B}\otimes\hat{A})\)
    from~\(\G[H]\) to~\(\DuG\);
  \item right quantum group homomorphisms \(\Delta_R\colon A\to
    A\otimes \hat{B}\);
  \item left quantum group homomorphisms \(\Delta_L\colon A\to \hat{B}\otimes A\);
  \item  the functor~\(F\) associated to~\(\Delta_R\) is the unique one that
  maps \((A,\Comult[A])\) to \((A,\Delta_R)\).  In general, \(F\) maps
  a continuous \(\G\)\nb-coaction \(\gamma\colon C\to C\otimes A\) to
  the unique \(\DuG[H]\)\nb-coaction \(\delta\colon C\to C\otimes
  \hat{B}\) for which the following diagram commutes:
  \begin{equation}
    \label{eq:right_quantum_group_homomorphism_as_fucntor}
    \begin{tikzpicture}[baseline=(current bounding box.west)]
      \matrix(m)[cd,column sep=4.5em]{
        C&C\otimes A\\
        C\otimes\hat{B}& C\otimes A\otimes \hat{B}\\
      };
      \draw[cdar] (m-1-1) -- node {\(\gamma\)} (m-1-2);
      \draw[cdar] (m-1-1) -- node[swap] {\(\delta\)} (m-2-1);
      \draw[cdar] (m-1-2) -- node {\(\Id_C\otimes\Delta_R\)} (m-2-2);
      \draw[cdar] (m-2-1) -- node[swap] {\(\gamma\otimes\Id_{\hat{B}}\)} (m-2-2);
    \end{tikzpicture}
    \qquad .
  \end{equation}  
  \end{enumerate}
  The first bijection maps a bicharacter~\(\bichar\) to its dual 
  \(\Dubichar\in\U(\hat{B}\otimes\hat{A})\) defined by
  \begin{equation}
    \label{eq:dual_bicharacter}
    \Dubichar\defeq\flip(\bichar^*).
  \end{equation}
  A bicharacter~\(\bichar\) and a right quantum group
  homomorphism~\(\Delta_R\) determine each other uniquely via
  \begin{equation}
    \label{eq:def_V_via_right_homomorphism}
    (\Id_{\hat{A}} \otimes \Delta_R)(\multunit[A])
    = \multunit[A]_{12}\bichar_{13}.
  \end{equation}
  Similarly, a bicharacter~\(\bichar\in\U(\hat{A}\otimes\hat{B})\)
  and a left quantum group homomorphisms~\(\Delta_L\) 
  determine each other uniquely by 
  \begin{equation}
    \label{eq:def_V_via_left_homomorphism}
    (\Id_{\hat{A}} \otimes \Delta_L)(\multunit[A])
    = \bichar_{12}\multunit[A]_{13}.
  \end{equation} 
 \end{theorem} 
 The dual bicharacter~\(\Dubichar\in\U(\hat{B}\otimes\hat{A})\)  
 describes the dual quantum group homomorphism 
 \(\hat{\Delta}_{R}\colon B\to B\otimes\hat{A}\). Thus~\(\Delta_R\) 
 and~\(\hat{\Delta}_{R}\) are in bijection as are~\(\bichar\) and~\(\Dubichar\). 
 A similar statement holds for~\(\Delta_{L}\) and~\(\hat{\Delta}_{L}\). 

\section{Heisenberg pairs revisited} 
 \label{sec:Heisenberg}
 Let \(\Qgrp{G}{A}\) and \(\Qgrp{H}{B}\) be \(\Cst\)\nb-quantum groups. 
 Let~\(\bichar\in\U(\hat{A}\otimes\hat{B})\) be a bicharacter.
 \begin{definition}[\cite{Meyer-Roy-Woronowicz:Twisted_tensor}*{Definition 3.1}]
  \label{def:V-Heisenberg_pair}
  A pair of representations \(\alpha\colon A\to\Bound(\Hils)\),
  \(\beta\colon B\to\Bound(\Hils)\) is called a
  \emph{\(\bichar\)\nb-Heisenberg pair}, or briefly
  \emph{Heisenberg pair}, if
  \begin{equation}
    \label{eq:V-Heisenberg_pair}
    \multunit[A]_{1\alpha}\multunit[B]_{2\beta}
    =\multunit[B]_{2\beta}\multunit[A]_{1\alpha} \bichar_{12}
    \qquad\text{in }\U(\hat{A}\otimes\hat{B}\otimes\Comp(\Hils));
  \end{equation}
  here \(\multunit[A]_{1\alpha} \defeq
  ((\Id_{\hat{A}}\otimes\alpha)\multunit[A])_{13}\) and
  \(\multunit[B]_{2\beta} \defeq
  ((\Id_{\hat{B}}\otimes\beta)\multunit[B])_{23}\).  It is
  called a \emph{\(\bichar\)\nb-anti-Heisenberg pair}, or
  briefly \emph{anti-Heisenberg pair}, if
  \begin{equation}
    \label{eq:V-anti-Heisenberg_pair}
    \multunit[B]_{2\beta}\multunit[A]_{1\alpha}
    =\bichar_{12}\multunit[A]_{1\alpha}\multunit[B]_{2\beta}
    \qquad\text{in \(\U(\hat{A}\otimes\hat{B}\otimes\Comp(\Hils))\),}
  \end{equation}
  with similar conventions as above.
  
  A~\(\bichar\)\nb-Heisenberg or~\(\bichar\)\nb-anti-Heisenberg pair~\((\alpha,\beta)\) is called 
  \emph{faithful} if the associated representations~\(\alpha\) and~\(\beta\) are faithful.
\end{definition} 

Recall that the unitary antipode \(\Coinv_A\colon A\to A\), is a linear,
involutive anti-automorphism (see Theorem~\ref{the:Cst_quantum_grp_and_mult_unit}). 
Given a pair of representations~\(\VHeisPair\) of~\(A\) and~\(B\) on~\(\Hils\) 
define the representations~\(\bar\alpha\colon A\to\Bound(\conj{\Hils})\) and 
\(\bar\beta\colon B\to\Bound(\conj{\Hils})\) by
\begin{align}
  \label{ex:exist_V-anti-Heisenberg_pair}   
  \bar\alpha(a)\defeq(\alpha(\Coinv_A(a)))^\transpose 
  &\quad&\text{and}&\quad &
  \bar\beta(b)&\defeq(\beta(\Coinv_B(b)))^\transpose.
\end{align}
Then~\cite{Meyer-Roy-Woronowicz:Twisted_tensor}*{Lemma 3.6} shows that 
\(\VHeisPair\) is a~\(\bichar\)\nb-Heisenberg pair on~\(\Hils\) if and only if  
\((\bar\alpha,\bar\beta)\)  as a~\(\bichar\)\nb-anti-Heisenberg pair 
on~\(\conj{\Hils}\).

In particular, assume that~\(\G\) and~\(\G[H]\) have bounded counits~\(e^{A}\colon A\to\C\) and 
\(e^{B}\colon B\to\C\), respectively. Then~\cite{Soltan-Woronowicz:Multiplicative_unitaries}*{Proposition 31} 
gives~\((\Id_{\hat{A}}\otimes e^{A})\multunit[A]= 1_{\hat{A}}\) and~\((\Id_{\hat{B}}\otimes e^{B})\multunit[B]
=1_{\hat{B}}\). Therefore \((e^{A},e^{B})\) is a \(\bichar\)\nb-Heisenberg and~\(\bichar\)\nb-anti-Heisenberg 
pair for~\(\bichar=1_{\hat{A}}\otimes 1_{\hat{B}}\in\U(\hat{A}\otimes\hat{B})\). Hence, in general,
a \(\bichar\)\nb-Heisenberg or \(\bichar\)\nb-anti-Heisenberg pair need not to be faithful. 
 
 When~\(\G=\DuG[H]\) and~\(\bichar=\multunit[A]\in\U(\hat{A}\otimes A)\),
 \(\multunit[A]\)\nb-Heisenberg pairs or \(\multunit[A]\)\nb-anti -Heisenberg pairs
 are also called~\(\G\)\nb-Heisenberg pairs or~\(\G\)\nb-anti-Heisenberg pairs, 
 respectively. Lemma~\textup{3.4} in~\cite{Meyer-Roy-Woronowicz:Twisted_tensor} 
 shows that a pair of representations~\((\pi,\hat{\pi})\) of \(A\)
  and~\(\hat{A}\) on~\(\Hils_{\pi}\) is a \(\G\)\nb-Heisenberg pair
  if and only if
  \begin{equation}
    \label{eq:Heisenberg_pair}
    \multunit[A]_{\hat{\pi}3}\multunit[A]_{1\pi}
    = \multunit[A]_{1\pi}\multunit[A]_{13}\multunit[A]_{\hat{\pi}3}
    \qquad\text{in }\U(\hat{A}\otimes\Comp(\Hils_{\pi})\otimes A).
  \end{equation}
    Here~\(\multunit[A]_{1\pi}\defeq((\Id_{\hat{A}}\otimes\pi)\multunit[A])_{12}\) and 
  \(\multunit[A]_{\hat{\pi}3}\defeq((\hat{\pi}\otimes\Id_{A})\multunit[A])_{23}\).

  Similarly, \((\rho,\hat{\rho})\) is a \(\G\)\nb-anti-Heisenberg pair on~\(\Hils_{\rho}\) 
  if and only if
  \begin{equation}
    \label{eq:anti-Heisenberg_pair}
    \multunit[A]_{1\rho}\multunit[A]_{\hat{\rho}3}
    = \multunit[A]_{\hat{\rho}3}\multunit[A]_{13}\multunit[A]_{1\rho}
    \qquad\text{in }\U(\hat{A}\otimes\Comp(\Hils_{\rho})\otimes A).
  \end{equation}
 Furthermore, Theorem~\ref{the:Cst_quantum_grp_and_mult_unit} and 
 \cite{Meyer-Roy-Woronowicz:Twisted_tensor}*{Lemma 3.6}  
 ensure that faithful \(\G\)\nb-Heisenberg and \(\G\)\nb-anti-Heisenberg pairs 
 exist. The following result is due to S.L.~Woronowicz by a private communication.
  \begin{proposition}
  \label{prop:faithful_property_of_Heisenberg_pairs}
  Every~\(\G\)\nb-Heisenberg pair or \(\G\)\nb-anti-Heisenberg pair is faithful.
 \end{proposition}
 
 To prove this, we first establish the following lemma.
 
 \begin{lemma}
  \label{lemm:Heisenberg_vs_anti_Heisenberg_1}
  Let \(\HeisPair{\pi}\) and \(\HeisPair{\rho}\) be a~\(\G\)\nb-Heisenberg pair and a 
  \(\G\)\nb-anti-Heisenberg pair on Hilbert spaces \(\Hils_{\pi}\) and \(\Hils_{\rho}\), respectively. 
  Then~\(\pi\otimes\hat{\rho}\colon A\otimes\hat{A}\to\Bound(\Hils_{\pi}\otimes\Hils_{\rho})\) and
  \(\rho\otimes\hat{\pi}\colon A\otimes\hat{A}\to\Bound(\Hils_{\rho}\otimes\Hils_{\pi})\) are unitarily equivalent. 
 \end{lemma}
 \begin{proof}
  Define \(\Psi\defeq\multunit_{\hat{\rho}\pi}\Flip\multunit_{\hat{\pi}\rho}\in\U(\Hils_{\pi}\otimes\Hils_{\rho},
  \Hils_{\rho}\otimes\Hils_{\pi})\), where~\(\multunit_{\hat{\pi}\rho}\defeq(\hat{\pi}\otimes\rho)
  \multunit\in\U(\Hils_{\pi}\otimes\Hils_{\rho})\), \(\multunit_{\hat{\rho}\pi}\defeq(\hat{\rho}\otimes\pi)
  \in\U(\Hils_{\rho}\otimes\Hils_{\pi})\), and~\(\Flip\colon\Hils_{\pi}\otimes\Hils_{\rho}\to\Hils_{\rho}\otimes\Hils_{\pi}\) 
  is the flip operator. We claim that~\(\Psi\) intertwines~\(\pi\otimes\hat{\rho}\) and
  \(\rho\otimes\hat{\pi}\). Using~\eqref{eq:first_leg_slice} and \eqref{eq:second_leg_slice}, 
  it suffices to show that
  \[
     \Psi_{23}\multunit_{1\pi}\multunit_{\hat{\rho}4}\Psi_{23}^{*}
   = \multunit_{1\rho}\multunit_{\hat{\pi}4} 
   \qquad\text{ in \(\U(\hat{A}\otimes\Comp(\Hils_{\rho})\otimes\Comp(\Hils_{\pi})\otimes A)\),}
  \]
  or, equivalently,
  \begin{equation}
   \label{eq:equiv_cond_Heisenberg_anti_Heisenberg_1}
     \Flip_{23}(\multunit_{\hat{\pi}\rho}\multunit_{1\pi}\multunit_{\hat{\rho}4}
     (\multunit_{\hat{\pi}\rho})^{*})\Flip_{23}
   = (\multunit_{\hat{\rho}\pi})^{*}\multunit_{1\rho}\multunit_{\hat{\pi}4}\multunit_{\hat{\rho}\pi}
  \end{equation}
  in~\(\U(\hat{A}\otimes\Comp(\Hils_{\rho})\otimes\Comp(\Hils_{\pi})\otimes A)\).

  The following computation yields~\eqref{eq:equiv_cond_Heisenberg_anti_Heisenberg_1}:
  \begin{align*}
    \Flip_{23}(\multunit_{\hat{\pi}\rho}\multunit_{1\pi}\multunit_{\hat{\rho}4}
     (\multunit_{\hat{\pi}\rho})^{*})\Flip_{23}
   &= \Flip_{23}(\multunit_{1\pi}\multunit_{1\rho}\multunit_{\hat{\pi}\rho}
      \multunit_{\hat{\rho}4}(\multunit_{\hat{\pi}\rho})^{*})\Flip_{23}\\
    &= \multunit_{1\pi}\multunit_{1\rho}
      \multunit_{\hat{\rho} 4}\multunit_{\hat{\pi}4}\\
    &= (\multunit_{\hat{\rho}\pi})^{*}\multunit_{1\rho}\multunit_{\hat{\rho}\pi}
      \multunit_{\hat{\rho} 4}\multunit_{\hat{\pi}4}
     = (\multunit_{\hat{\rho}\pi})^{*}\multunit_{1\rho}
      \multunit_{\hat{\pi}4}\multunit_{\hat{\rho}\pi};
  \end{align*}
  the first equality uses~\eqref{eq:Heisenberg_pair}, the second equality uses~\eqref{eq:anti-Heisenberg_pair}
  and an application of \(\Flip_{23}\), the third equality again uses~\eqref{eq:anti-Heisenberg_pair}, and the fourth
  equality uses~\eqref{eq:Heisenberg_pair}.
 \end{proof}
 
 \begin{proof}[Proof of Proposition~\textup{\ref{prop:faithful_property_of_Heisenberg_pairs}}]
  Let \(\HeisPair{\pi}\) and~\(\HeisPair{\rho}\) be~\(\G\)\nb-Heisenberg and anti-Heisenberg
  pairs on~\(\Hils_{\pi}\) and~\(\Hils_{\rho}\), respectively. Lemma~\ref{lemm:Heisenberg_vs_anti_Heisenberg_1}
  forces~\(\pi\otimes\hat{\rho}\) and~\(\rho\otimes\hat{\pi}\) to be unitarily equivalent. 
  By~\cite{Dixmier:Cstar-algebras}*{Proposition 5.3}, the representations~\(\pi\) and
  \(\rho\) of~\(A\) on~\(\Hils_{\pi}\) and~\(\Hils_{\rho}\) are quasi-equivalent.
  Therefore there is a unique quasi-equivalence class of representations of~\(A\) that 
  contains the first element of all \(\G\)\nb-Heisenberg and \(\G\)\nb-anti-Heisenberg
  pairs. Therefore, \(\pi\) and~\(\rho\) are quasi equivalent to the faithful representation 
  of~\(A\) in Theorem~\ref{the:Cst_quantum_grp_and_mult_unit}, hence they are 
  faithful.  Similarly, \(\hat{\rho}\) and~\(\hat{\pi}\) are quasi-equivalent representations 
  of~\(\hat{A}\) on~\(\Hils_{\pi}\) and \(\Hils_{\rho}\), respectively. A similar 
  argument shows gives~\(\hat{\pi}\) and~\(\hat{\rho}\) are also faithful.
\end{proof}

 The character condition~\eqref{eq:W_char_in_second_leg} and the pentagon equation~\eqref{eq:pentagon} 
 yield \((\Id_{\hat{A}}\otimes(\pi\otimes\Id_{A})\Comult[A])\multunit[A]=\multunit[A]_{1\pi}\multunit[A]_{13}=
 \multunit[A]_{\hat{\pi}3}\multunit[A]_{1\pi}(\multunit[A]_{\hat{\pi}3})^{*}\) in 
 \(\U(\hat{A}\otimes\Comp(\Hils)\otimes A)\), where~\((\pi,\hat{\pi})\) is a 
 \(\G\)\nb-Heisenberg pair on a Hilbert space~\(\Hils\). 
 Slicing the first leg by~\(\omega\in\hat{A}'\) and using 
 \eqref{eq:first_leg_slice} we get 
 \begin{equation}
  \label{eq:unit_imp_comult}
  (\pi\otimes\Id_{A})\Comult[A](a)
  =(\multunit[A]_{\hat{\pi}2})(\pi(a)\otimes 1)(\multunit[A]_{\hat{\pi}2})^{*} 
  \qquad\text{for all~\(a\in A\).}
 \end{equation}
  Since~\(\pi\) is faithful, this says that~\(\Comult[A]\) is 
 \emph{implemented} by~\(\multunit[A]\). Indeed, this is a 
 well known fact in the theory of locally compact quantum 
 groups (e.g. see~\cite{Soltan-Woronowicz:Multiplicative_unitaries}).

 Lemma~\textup{3.8} in~\cite{Meyer-Roy-Woronowicz:Twisted_tensor} provides 
 one way to construct faithful~\(\bichar\)\nb-Heisenberg pairs. 
 A similar argument gives the following corollary
\begin{corollary}
  \label{cor:existence_of_canonical_Heisenberg_pair}
  Let~\(\HeisPair{\pi}\) be a \(\G\)\nb-Heisenberg pair on a Hilbert space
  \(\Hils\) and let~\(\eta\colon B\to\Bound(\Hils[K])\) be a faithful representation 
  of  \(B\) on~\(\Hils[K]\).  Then the pair of representations
  \((\alpha,\beta)\) of \(A\) and~\(B\) on
  \(\Hils[K]\otimes\Hils\) defined by \(\alpha(a) \defeq
  1_{\Hils[K]}\otimes\pi(a)\) and \(\beta(b) \defeq
  (\eta\otimes\hat{\pi})\hat{\Delta}_{R}(b)\) is a faithful \(\bichar\)\nb-Heisenberg pair.
\end{corollary}

\section{Drinfeld pairs}
  \label{sec:Drinf}
 Let \(\Qgrp{G}{A}\) and~\(\Qgrp{H}{B}\) be \(\Cst\)\nb-quantum
 groups.  Let \(\multunit[A]\in\U(\hat{A}\otimes A)\) and
 \(\multunit[B]\in\U(\hat{B}\otimes B)\) be their reduced
 bicharacters.  Let \(\bichar\in\U(\hat{A}\otimes\hat{B})\)
 be a bicharacter from \(\G\) to~\(\DuG[H]\).    
 \begin{definition}
 \label{def:V-Drinfeld_comm}
 A pair~\((\rho,\theta)\) of representations of~\(A\) and 
 \(B\) on a Hilbert space \(\Hils\) is a 
 \emph{\(\bichar\)\nb-Drinfeld pair} if 
 \begin{equation}
  \label{eq:V-Drinfeld}
    \bichar_{12}\multunit[A]_{1\rho}\multunit[B]_{2\theta}
  = \multunit[B]_{2\theta}\multunit[A]_{1\rho}\bichar_{12}
  \qquad\text{ in~\(\U(\hat{A}\otimes\hat{B}\otimes\Comp(\Hils))\).}
 \end{equation}
 A~\(\bichar\)\nb-Drinfeld pair~\((\rho,\theta)\) is~\emph{faithful} if the associated 
 representations~\(\rho\) and~\(\theta\) are faithful.
\end{definition}

\begin{example}
 \label{ex:V-Drinfeld_group_case} 
 Let~\(G\) and~\(H\) be locally compact groups and let 
 \(A=\Cred(G)\) and~\(B=\Cred(H)\) be the associated reduced quantum 
 groups. Then every bicharacter~\(\bichar\in\U(\hat{A}\otimes\hat{B})\) 
 is indeed a continuous bicharacter on the group \(G\times H\). Hence, 
 any pair of commuting representations 
 \(\rho\colon\Cred(G)\to\Bound(\Hils)\) and 
 \(\theta\colon\Cred(H)\to\Bound(\Hils)\) satisfy 
 \eqref{eq:V-Drinfeld} independent of the choice of bicharacters.
\end{example}  

\begin{example}
 \label{ex:V-Drinf_qnt_grp_case}
 Let~\(\hat{B}=A\), \(\DuComult[B]=\Comult[A]\) and~\(\bichar=\multunit[A]
 \in\U(\hat{A}\otimes A)\). We call~\(\multunit[A]\)\nb-Drinfeld pairs \emph{\(\G\)\nb-Drinfeld pairs}. 
 A pair of representations \(\rho\colon A\to\Bound(\Hils)\) and~\(\theta\colon\hat{A}\to\Bound(\Hils)\) 
 is a~\(\G\)\nb-Drinfeld pair if and only if it satisfies the 
 \(\G\)\nb-\emph{Drinfeld commutation relation}:
 \begin{equation}
  \label{eq:G-Drinfeld}
   \multunit[A]_{1\rho}\multunit[A]_{13}\multunit[A]_{\theta3}
 = \multunit[A]_{\theta3}\multunit[A]_{13}\multunit[A]_{1\rho}
 \qquad\text{in~\(\U(\hat{A}\otimes\Comp(\Hils)\otimes A)\).}
 \end{equation}
 
 Define~\(\Rmat\defeq(\theta\otimes\rho)\multunit[A]\in\U(\Hils\otimes\Hils)\). 
 Equation~\eqref{eq:G-Drinfeld} says that~\(\Rmat\) is a solution to the 
 \emph{Yang--Baxter Equation}: 
 \begin{equation}
  \label{eq:Yang_baxter}
     \Rmat_{12}\Rmat_{13}\Rmat_{23}
  = \Rmat_{23}\Rmat_{13}\Rmat_{12} 
  \qquad\text{in~\(\U(\Hils\otimes\Hils\otimes\Hils)\).}
 \end{equation}    
\end{example}

 Theorem~\ref{the:equivalent_notion_of_homomorphisms} shows that 
 a bicharacter~\(\bichar\in\U(\hat{A}\otimes\hat{B})\) naturally gives rise to a 
 dual bicharacter~\(\Dubichar\in\U(\hat{B}\otimes\hat{A})\), a right quantum 
 group homomorphism~\(\Delta_R\colon A\to A\otimes\hat{B}\), and a 
 left quantum group homomorphism~\(\Delta_L\colon A\to\hat{B}\otimes A\). 
 This leads us to reformulate the condition of being a \(\bichar\)\nb-Drinfeld pair 
 in the following way: 
\begin{lemma}
  \label{lemm:equiv_cond_V_Drinfeld_pair}
  Let \(\rho\) and~\(\theta\) be representations of \(A\) and \(B\)
  on a Hilbert space~\(\Hils\).  
  Then the following are equivalent:
  \begin{enumerate}
  \item \((\rho,\theta)\) is a \(\bichar\)\nb-Drinfeld pair;
  \item \((\theta,\rho)\) is a \(\Dubichar\)\nb-Drinfeld pair;
  \item \((\Id_{\hat{B}}\otimes\rho)\Delta_L(a)
         =(\multunit[B]_{1\theta})\big((\Id_{\hat{B}}\otimes\rho)\flip\Delta_R(a)\big) 
            (\multunit[B]_{1\theta})^{*}\) for all~\(a\in A\).
  \item \((\Id_{\hat{A}}\otimes\theta)\hat{\Delta}_{L}(b)
         =(\multunit[A]_{1\rho})\big((\Id_{\hat{A}}\otimes\theta)\flip\hat{\Delta}_R(b)\big) 
            (\multunit[A]_{1\rho})^{*}\) for all~\(b\in B\).
  \end{enumerate}
\end{lemma}

\begin{proof}
  (1)\(\iff\)(2): (1) is equivalent to
  \[
  \multunit[A]_{1\rho}\multunit[B]_{2\theta}\bichar^*_{12}
  =\bichar_{12}^*\multunit[B]_{2\theta}\multunit[A]_{1\rho}
  \qquad \text{in }\U(\hat{A}\otimes\hat{B}\otimes\Comp(\Hils))
  \]
  by~\eqref{eq:V-Drinfeld}.  Applying~\(\flip_{12}\) gives
  \begin{align}
    \label{eq:hat-V-Drinfeld_pair}
    \multunit[A]_{2\rho}\multunit[B]_{1\theta}\Dubichar_{12}
    &= \Dubichar_{12}\multunit[B]_{1\theta}\multunit[A]_{2\rho}
    \qquad \text{in }\U(\hat{B}\otimes\hat{A}\otimes\Comp(\Hils)),
  \end{align}
  which is equivalent to \((\theta,\rho)\) being a
  \(\Dubichar\)\nb-Drinfeld pair.  Thus (1)\(\iff\)(2).

  (1)\(\iff\)(3): Let~\((\rho,\theta)\) be a~\(\bichar\)\nb-Drinfeld 
  pair.  The following computation takes place in
  \(\U(\hat{A}\otimes\hat{B}\otimes\Comp(\Hils))\):
  \begin{align*}
    (\Id_{\hat{A}}\otimes\Id_{\hat{B}}\otimes\rho)
    (\Id_{\hat{A}}\otimes\Delta_L)\multunit[A]
    &= \bichar_{12}\multunit[A]_{1\rho}
    = (\multunit[B]_{2\theta})\multunit[A]_{1\rho}\bichar_{12}(\multunit[B]_{2\theta})^{*} \\
    &= (\multunit[B]_{2\theta})
       \big(\Id_{\hat{A}}\otimes((\Id_{\hat{B}}\otimes\rho)\flip\Delta_R)\big)
       (\multunit[B]_{2\theta})^{*}.
  \end{align*}
  The first equality uses~\eqref{eq:def_V_via_left_homomorphism}; the
  second equality uses~\eqref{eq:V-Drinfeld}; and the third equality
  uses~\eqref{eq:def_V_via_right_homomorphism}. Since
  \(\{(\omega\otimes\Id_A)\multunit[A]:\omega\in\hat{A}'\}\) is
  linearly dense in~\(A\), slicing the first leg of the first and the
  last expression in the above equation shows that
  (1)\(\Longrightarrow\)(3).

  Conversely, applying~\(\Id_{\hat{A}}\otimes\Id_{\hat{A}}\otimes\rho\) on both sides
  of~\eqref{eq:def_V_via_left_homomorphism} and using~(4), we get
  \[
  \bichar_{12}\multunit[A]_{1\rho}
  = (\Id_{\hat{A}}\otimes(\Id_{\hat{B}}\otimes\rho)\Delta_L)\multunit[A]
  = (\multunit[B]_{2\theta})\multunit[A]_{1\rho}\bichar_{12}(\multunit[B]_{2\theta})^*
  \quad \text{in }\U(\hat{A}\otimes\hat{B}\otimes\Comp(\Hils)),
  \]
  which is equivalent to~\eqref{eq:V-Drinfeld}. Thus (3)\(\Longrightarrow\)(1).

  To prove (2)\(\iff\)(4), argue as in the proof that (1)\(\iff\)(3).
\end{proof}

\subsection{Heisenberg pair versus Drinfeld pair}
 Certain ways of putting Heisenberg and anti-Heisenberg pairs together give 
 ordinary commutation (see~\cite{Meyer-Roy-Woronowicz:Twisted_tensor}*{Proposition 3.9}). 
 This played a crucial role for the construction of 
 twisted tensor products of~\(\Cst\)\nb-algebras 
 in~\cite{Meyer-Roy-Woronowicz:Twisted_tensor}. 
 Changing their order yields the  
 next proposition, which ensures the existence 
 of~\(\bichar\)\nb-Drinfeld pairs.
  
\begin{proposition}
 \label{prop:V-drinfeld_comm}
 Let~\(\VHeisPair\) and~\((\bar\alpha,\bar\beta)\) be a 
 \(\bichar\)\nb-Heisenberg and \(\bichar\)\nb-anti-Heisenberg pair 
 on~\(\Hils\) and~\(\Hils[K]\), respectively. Define the representations
 \(\rho\defeq (\bar\alpha\otimes\alpha)\Comult[A]\) and
 \(\theta\defeq (\bar\beta\otimes\beta)\Comult[B]\) of 
 \(A\) and~\(B\) on~\(\Hils[K]\otimes\Hils\). Then \((\rho,\theta)\) 
 is a~\(\bichar\)\nb-Drinfeld pair on~\(\Hils[K]\otimes\Hils\).
\end{proposition}
\begin{proof}
 We must check~\eqref{eq:V-Drinfeld} for~\((\rho,\theta)\).
 The character condition~\eqref{eq:W_char_in_second_leg} for~\(\multunit[A]\) and 
 \(\multunit[B]\) gives:
\[
   \multunit[A]_{1\rho}=\multunit[A]_{1\bar\alpha}\multunit[A]_{1\alpha}
   \quad\text{and}\quad
   \multunit[B]_{2\theta}=\multunit[B]_{2\bar\beta}\multunit[B]_{2\beta}
   \quad\text{in~\(\U(\hat{A}\otimes\hat{B}\otimes\Comp(\Hils[K]\otimes\Hils))\)}.
 \]  
 Clearly,~\(\multunit[A]_{1\alpha}\) commutes with \(\multunit[B]_{2\bar\beta}\) 
 and \(\multunit[A]_{1\bar\alpha}\) commutes with \(\multunit[B]_{2\beta}\) 
 inside~\(\U(\hat{A}\otimes\hat{B}\otimes\Comp(\Hils[K]\otimes\Hils))\).
 The defining conditions~\eqref{eq:V-Heisenberg_pair} and~\eqref{eq:V-anti-Heisenberg_pair} 
 of \(\bichar\)\nb-Heisenberg and \(\bichar\)\nb-anti-Heisenberg pairs give
 \begin{align*}
    \bichar_{12}\multunit[A]_{1\rho}\multunit[B]_{2\theta}
  = \bichar_{12}\multunit[A]_{1\bar\alpha}\multunit[A]_{1\alpha}
    \multunit[B]_{2\bar\beta}\multunit[B]_{2\beta}
  &= \bichar_{12}\multunit[A]_{1\bar\alpha}\multunit[B]_{2\bar\beta}
    \multunit[A]_{1\alpha}\multunit[B]_{2\beta}\\
  &= \multunit[B]_{2\bar\beta}\multunit[A]_{1\bar\alpha}
    \multunit[B]_{2\beta}\multunit[A]_{1\alpha}\bichar_{12}\\
  &= \multunit[B]_{2\bar\beta}\multunit[B]_{2\beta}
        \multunit[A]_{1\bar\alpha}\multunit[A]_{1\alpha}\bichar_{12}\\
   &= \multunit[B]_{2\theta}\multunit[A]_{1\rho}\bichar_{12}.\qedhere 
 \end{align*}
\end{proof}
%
%
 
\subsection{From Drinfeld pairs to multiplicative unitaries}
 \label{subsec:Multunit} 
 The goal of this subsection is to systematically construct 
 a modular or manageable multiplicative unitary associated 
 to certain \(\bichar\)\nb-Drinfeld pairs. 
   
  Let~\(\Hils\) be a Hilbert space, and let 
  \(\Multunit[A]\in\U(\Hils\otimes\Hils)\) 
  be a manageable multiplicative unitary that generates~\(\G\). 
  By Theorem~\ref{the:Cst_quantum_grp_and_mult_unit}, there is a 
  \(\G\)\nb-Heisenberg pair \(\HeisPair{\pi}\) on \(\Hils\) 
  such that \(\Multunit[A]=(\hat{\pi}\otimes\pi)\multunit[A]\).
  
 Similarly, let~\(\Hils[K]\) be a Hilbert space, 
 and let \(\Multunit[B]\in\U(\Hils[K]\otimes\Hils[K])\) 
  be a manageable multiplicative unitary generating~\(\G[H]\). 
  Let \(\HeisPair{\eta}\) be the corresponding~\(\G[H]\)\nb-Heisenberg 
  pair \(\Hils[K]\) such that~\(\Multunit[B]=(\hat{\eta}\otimes\eta)\multunit[B]
  \in\U(\Hils[K]\otimes\Hils[K])\).
  
  Proposition~\ref{prop:faithful_property_of_Heisenberg_pairs} shows that 
  the representations~\(\pi\), \(\hat{\pi}\), \(\eta\) and~\(\hat{\eta}\) are 
  faithful. Hence, the \(\bichar\)\nb-Heisenberg pair~\(\VHeisPair\) on 
  \(\Hils[K]\otimes\Hils\) in Corollary~\ref{cor:existence_of_canonical_Heisenberg_pair} 
  is also faithful. Using~\eqref{ex:exist_V-anti-Heisenberg_pair} we construct the 
  associated faithful \(\bichar\)\nb-anti-Heisenberg pair \((\bar\alpha,\bar\beta)\) 
  on~\(\conj{\Hils[K]}\otimes\conj{\Hils}\).
  
  Define the representations of~\(A\), \(B\), \(\hat{A}\), \(\hat{B}\) on the 
  Hilbert space 
  \(\Hils_{\mathcal{D}}\defeq\conj{\Hils[K]}\otimes\conj{\Hils}\otimes\Hils[K]\otimes\Hils\) by:
   \begin{equation}
    \label{eq:gen_codob_reps}
     \begin{alignedat}{2}
       \rho(a) &\defeq (\bar{\alpha}\otimes\alpha)\Comult[A](a) 
      &\qquad&\text{for all~\(a\in A\),}\\
       \theta(b) &\defeq (\bar{\beta}\otimes\beta)\Comult[B](b) 
      &\qquad&\text{for all~\(b\in B\),}\\
       \xi(\hat{a}) &\defeq 1_{\conj{\Hils[K]}\otimes\conj{\Hils}}\otimes 1_{\Hils[K]}\otimes\hat{\pi}(\hat{a})
      &\qquad&\text{for all~\(\hat{a}\in\hat{A}\),}\\
       \zeta(\hat{b}) &\defeq 1_{\conj{\Hils[K]}\otimes\conj{\Hils}}\otimes\hat{\eta}(\hat{b})\otimes 1_{\Hils}
      &\qquad&\text{for all~\(\hat{b}\in\hat{B}\).}  
    \end{alignedat}
   \end{equation} 
   Let us denote~\((\xi\otimes\rho)\multunit[A]\in\U(\Hils_{\mathcal{D}}\otimes\Hils_{\mathcal{D}})\) 
   and~\((\zeta\otimes\theta)\multunit[B]\in\U(\Hils_{\mathcal{D}}\otimes\Hils_{\mathcal{D}})\) by 
   \(\multunit[A]_{\xi\rho}\) and~\(\multunit[B]_{\zeta\theta}\), respectively.
   \begin{theorem}
    \label{the:multunit_drinf}
    The unitary \(\Multunit[\mathcal{D}]\defeq\multunit[A]_{\xi\rho}\multunit[B]_{\zeta\theta} 
    \in\U(\Hils_{\mathcal{D}}\otimes\Hils_{\mathcal{D}})\) is a modular multiplicative 
    unitary.
   \end{theorem} 
 The dual of a modular multiplicative unitary is again modular 
 (see~\cite{Soltan-Woronowicz:Remark_manageable}*{Proposition 2.2}).
 Hence, it is equivalent to show that 
 \(\DuMultunit[\mathcal{D}]\defeq\Dumultunit[B]_{\theta\zeta}\Dumultunit[A]_{\rho\xi}
 \in\U(\Hils_{\mathcal{D}}\otimes\Hils_{\mathcal{D}})\) is a modular multiplicative unitary. 
 The next result is the first step towards the proof of this fact.
 \begin{proposition}
  \label{prop:multunit_gen_codob} 
  \(\DuMultunit[\mathcal{D}]\in\U(\Hils_{\mathcal{D}}\otimes\Hils_{\mathcal{D}})\) 
  is a multiplicative unitary.
 \end{proposition}
 To prove this, we need to understand the commutation relations between 
 the representations in \eqref{eq:gen_codob_reps}. 
 \begin{lemma}
   \label{lemm:commutation_bet_diff_reps}
   Consider the faithful representations on~\(\Hils_{\mathcal{D}}\) defined in~\eqref{eq:gen_codob_reps}. 
   Then 
   \begin{enumerate}
    \item\label{eq:r_u_comm} \((\rho,\xi)\) is a~\(\G\)\nb-Heisenberg pair;
    \item\label{eq:r_s_comm} \((\rho,\theta)\) is a~\(\bichar\)\nb-Drinfeld pair;
    \item\label{eq:s_v_comm} \(\theta\) and~\(\zeta\) commute in the following way:
         \begin{equation}
           \label{eq:s_and_v_comm}
             \Dumultunit[B]_{\theta3}\Dumultunit[B]_{1\zeta} 
           = \Dumultunit[B]_{1\zeta}\bichar_{\xi3}\Dumultunit[B]_{13}
              \bichar_{\xi3}^{*}\Dumultunit[B]_{\theta3}
           \qquad\text{in~\(\U(B\otimes\Comp(\Hils_{\mathcal{D}})\otimes\hat{B})\);}
         \end{equation}
    \item\label{eq:s_u_comm} \(\theta\) and~\(\xi\) commute in the following way:
          \begin{equation}
           \label{eq:s_and_u_comm}
             \Dumultunit[B]_{\theta3}\Dumultunit[A]_{1\xi}
           = \bichar_{\xi3}\Dumultunit[A]_{1\xi}\bichar_{\xi3}^{*}\Dumultunit[B]_{\theta3}
           \quad\text{in~\(\U(A\otimes\Comp(\Hils_{\mathcal{D}})\otimes\hat{B})\);} 
          \end{equation}
   \item\label{eq:r_v_comm} \(\rho\) and~\(\zeta\) commute;
   \item\label{eq:u_v_comm} \(\xi\) and~\(\zeta\) commute.
   \end{enumerate}
  \end{lemma}
  \begin{proof}
   Corollary~\ref{cor:existence_of_canonical_Heisenberg_pair} gives 
   \(\rho(a)=\big((\bar\alpha\otimes\pi)\Comult[A](a)\big)_{14}\in\Bound(\Comp(\Hils_{\mathcal{D}})\otimes
   \conj{\Hils[K]}\otimes\conj{\Hils}\otimes\Hils[K]\otimes\Hils)\); hence 
   \cite{Meyer-Roy-Woronowicz:Twisted_tensor}*{Lemma 3.8} gives 
   \ref{eq:r_u_comm}. Proposition~\ref{prop:V-drinfeld_comm}  yields 
   \ref{eq:r_s_comm}. Also, \ref{eq:r_v_comm} and~\ref{eq:u_v_comm} follow 
   from~\eqref{eq:gen_codob_reps}.
   
   We express \(\hat{\Delta}_{R}\) in terms of~\(\Dubichar\),  following 
   \eqref{eq:def_V_via_right_homomorphism}:
   \[
    (\Id_{\hat{B}}\otimes\hat{\Delta}_R)\multunit[B]=\multunit[B]_{12}\Dubichar_{13}
    \qquad\text{in~\(\U(\hat{B}\otimes B\otimes\hat{A})\).}
   \]
   Applying~\(\flip_{23}\flip_{12}\) to the both sides of the last 
   expression and taking adjoints yields
   \begin{equation}
    \label{eq:dual_right_morph_WB}
     (\hat{\Delta}_{R}\otimes\Id_{\hat{B}})\Dumultunit[B] =\bichar_{23}\Dumultunit[B]_{13}
     \qquad\text{in~\(\U(B\otimes\hat{A}\otimes\hat{B})\).}
   \end{equation}
   By definition,~\(\theta(b)=(\bar{\beta}\otimes ((\eta\otimes\hat{\pi})\hat{\Delta}_R)\otimes\Id_{\hat{B}})\Comult[B]\). 
   The character condition~\eqref{eq:W_char_in_first_leg} for ~\(\Dumultunit[B]\) gives
   \begin{equation}
    \label{eq:du_mult_H_s}
        \Dumultunit[B]_{\theta4}
      = (\bar{\beta}\otimes ((\eta\otimes\hat{\pi})\hat{\Delta}_R)\otimes\Id_{\hat{B}})
        \bigl(\Dumultunit[B]_{23}\Dumultunit[B]_{13}\bigr) 
      = \bichar_{\hat{\pi}4}\cdot\Dumultunit[B]_{\eta 4}\Dumultunit[B]_{\bar{\beta}4}
   \end{equation}
   in~\(\U(\Comp(\conj{\Hils[K]}\otimes\conj{\Hils}\otimes\Hils[K]\otimes\Hils)\otimes\hat{B})\).
   Using~\eqref{eq:du_mult_H_s},  we get
   \[
     \Dumultunit[B]_{\theta5}\Dumultunit[B]_{1\zeta} 
   = \bichar_{\hat{\pi}5}\Dumultunit[B]_{\eta 5}\Dumultunit[B]_{\bar{\beta}5}\Dumultunit[B]_{1\hat{\eta}}
  \qquad\text{in~\(\U(B\otimes\Comp(\conj{\Hils[K]}\otimes\conj{\Hils}\otimes\Hils[K]\otimes\Hils)\otimes\hat{B})\).}
   \] 
   Since~\(\Dumultunit[B]_{\bar{\beta}5}\) and~\(\Dumultunit[B]_{1\hat{\eta}}\) commute,   
   \(\HeisPair{\eta}\) is an \(\G[H]\)\nb-Heisenberg pair, and~\(\bichar_{\hat{\pi}5}\) and 
   \(\Dumultunit[B]_{1\hat{\eta}}\Dumultunit[B]_{15}\) commute, we get
   \[
    \Dumultunit[B]_{\theta5}\Dumultunit[B]_{1\zeta}
     = \bichar_{\hat{\pi}5}\Dumultunit[B]_{\eta 5}\Dumultunit[B]_{1\hat{\eta}}\Dumultunit[B]_{\bar{\beta}5}
    = \bichar_{\hat{\pi}5}\Dumultunit[B]_{1\hat{\eta}}\Dumultunit[B]_{15}\Dumultunit[B]_{\eta 5} 
    \Dumultunit[B]_{\bar{\beta}5}
    = \Dumultunit[B]_{1\hat{\eta}}\bichar_{\hat{\pi}5}\Dumultunit[B]_{15}\Dumultunit[B]_{\eta 5} 
    \Dumultunit[B]_{\bar{\beta}5} 
   \]
   in~\(\U(B\otimes\Comp(\conj{\Hils[K]}\otimes\conj{\Hils}\otimes\Hils[K]\otimes\Hils)\otimes\hat{B})\). 
   Hence~\eqref{eq:s_and_v_comm} follows from~\eqref{eq:du_mult_H_s} 
   and~\eqref{eq:gen_codob_reps}, after collapsing the respective legs under the identification 
   \(\Hils_{\mathcal{D}}=\conj{\Hils[K]}\otimes\conj{\Hils}\otimes\Hils[K]\otimes\Hils\).
   
   Similarly,~\eqref{eq:s_and_u_comm} follows from~\eqref{eq:du_mult_H_s} and~\eqref{eq:gen_codob_reps} 
   after collapsing the leg numbers:
   \[
       \Dumultunit[B]_{\theta3}\Dumultunit[A]_{1\xi}
     = \bichar_{\hat{\pi}5}\Dumultunit[B]_{\eta 5}\Dumultunit[B]_{\bar{\beta}5}\Dumultunit[A]_{1\hat{\pi}}
     = \bichar_{\hat{\pi}5}\Dumultunit[A]_{1\hat{\pi}}\Dumultunit[B]_{\eta 5}\Dumultunit[B]_{\bar{\beta}5}
     = \bichar_{\xi3}\Dumultunit[A]_{1\xi}\bichar^{*}_{\xi3}\Dumultunit[B]_{\theta3}.\qedhere
   \] 
  \end{proof}
 \begin{notation}
  \label{not:leg_numbering}
    We write~\(\pi_i\) when a representation \(\pi\) is acting
    on the \(i\)th leg of a unitary.
 \end{notation}    
  
\begin{proof}[Proof of Proposition~\textup{\ref{prop:multunit_gen_codob}}]
   Rewrite~\eqref{eq:V-Drinfeld} for~\((\rho,\theta)\) involving~\(\Dumultunit[A]\) 
   and~\(\Dumultunit[B]\) in the following way: 
   \begin{equation}
    \label{eq:rho_theta_equiv_Drinfeld_cond}
     \Dumultunit[A]_{\rho2}\Dumultunit[B]_{\theta3}\bichar_{23}
    =\bichar_{23}\Dumultunit[B]_{\theta3}\Dumultunit[A]_{\rho2}
    \qquad\text{in~\(\U(\Comp(\Hils_{\mathcal{D}})\otimes\hat{A}\otimes\hat{B})\).}
   \end{equation}

   Equations~\eqref{eq:rho_theta_equiv_Drinfeld_cond} and~\eqref{eq:s_and_u_comm} give:
   \begin{equation}
    \label{eq:aux_gen_co_dob_mult}
      \Dumultunit[A]_{\rho_1 \xi_2}\Dumultunit[B]_{\theta_1 3}\Dumultunit[B]_{\theta_2 3} 
    = \bichar_{\xi_2 3}\Dumultunit[B]_{\theta_1 3}\Dumultunit[A]_{\rho_1 \xi_2}
       \bichar_{\xi_2 3}^{*}\Dumultunit[B]_{\theta_2 3}
    = \bichar_{\xi_2 3}\Dumultunit[B]_{\theta_1 3}\bichar_{\xi_2 3}^{*}
      \Dumultunit[B]_{\theta_2 3}\Dumultunit[A]_{\rho_1 \xi_2}
   \end{equation}
   in~\(\U(\Comp(\Hils_{\mathcal{D}}\otimes\Hils_{\mathcal{D}})\otimes\hat{B})\).
   
  The following computation takes place in \(\U(\Comp(\Hils_{\mathcal{D}}\otimes\Hils_{\mathcal{D}})\otimes\hat{B}\otimes\hat{A})\):   
  \begin{align*}
  &   \Dumultunit[B]_{\theta_2 3}\Dumultunit[A]_{\rho_2 4}\Dumultunit[B]_{\theta_1 \zeta_2}
     \Dumultunit[A]_{\rho_1 \xi_2}(\Dumultunit[A]_{\rho_2 4})^{*}(\Dumultunit[B]_{\theta_2 3})^{*}\\
  &= \Dumultunit[B]_{\theta_2 3}\Dumultunit[B]_{\theta_1 \zeta_2}\Dumultunit[A]_{\rho_2 4}\Dumultunit[A]_{\rho_1 \xi_2}
     (\Dumultunit[A])^{*}_{\rho_2 4}(\Dumultunit[B])^{*}_{\theta_2 3}\\
  &= \Dumultunit[B]_{\theta_1 \zeta_2}\bichar_{\xi_2 3}\Dumultunit[B]_{\theta_1 3}
     \bichar_{\xi_2 3}^{*}\Dumultunit[B]_{\theta_2 3}\Dumultunit[A]_{\rho_1 \xi_2}
     \Dumultunit[A]_{\rho_1 4}(\Dumultunit[B])^{*}_{\theta_2 3}\\
  &= \Dumultunit[B]_{\rho_1 \zeta_2}\Dumultunit[A]_{\rho_1 \xi_2}\Dumultunit[B]_{\theta_1 3}\Dumultunit[B]_{\theta_2 3}
     (\Dumultunit[B])^{*}_{\theta_2 3}\Dumultunit[A]_{\rho_1 4}\\
  &=  \Dumultunit[B]_{\theta_1 \zeta_2}\Dumultunit[A]_{\rho_1 \xi_2}\Dumultunit[B]_{\theta_1 3}\Dumultunit[A]_{\rho_1 4}.   
  \end{align*}
  The first equality follows from Lemma~\ref{lemm:commutation_bet_diff_reps} \ref{eq:r_v_comm},
  the second equality uses~\eqref{eq:s_and_v_comm} and Lemma~\ref{lemm:commutation_bet_diff_reps} 
  \ref{eq:r_u_comm}, the third equality uses~\eqref{eq:aux_gen_co_dob_mult} and 
  that~\(\Dumultunit[B]_{\theta_2 3}\), \(\Dumultunit[A]_{\rho_1 4}\) commute, and the last equality is trivial.
  
  Finally, applying \(\xi\) and~\(\zeta\) on the third and fourth leg on both sides of the last 
  expression gives the pentagon equation~\eqref{eq:pentagon} for \(\DuMultunit[\mathcal{D}]\in\U(\Hils_{\mathcal{D}}\otimes\Hils_{\mathcal{D}})\).
\end{proof}

 Next we need to know what it means for~\(\Bichar\defeq(\hat{\pi}\otimes \hat{\eta})
 \bichar\in\U(\Hils\otimes\Hils[K])\) to be manageable. 
 By \cite{Meyer-Roy-Woronowicz:Homomorphisms}*{Lemma 3.2},  
 \(\Bichar\in\U(\Hils\otimes\Hils[K])\) is adapted to~\(\DuMultunit[B]\in\U(\Hils[K]\otimes\Hils[K])\) in 
 the sense of \cite{Woronowicz:Multiplicative_Unitaries_to_Quantum_grp}*{Definition 1.3}. 

 \cite{Woronowicz:Multiplicative_Unitaries_to_Quantum_grp}*{Theorem 1.6} 
 gives the manageability of \(\Bichar\in\U(\Hils\otimes\Hils[K])\): there is a unitary 
 \(\widetilde{\Bichar}\in\U(\conj{\Hils}\otimes\Hils[K])\) with the following 
 condition:
  \begin{equation}
   \label{eq:second_leg_bichar_adapted}
    \big(x\otimes u\mid \Bichar\mid z\otimes y\big)
    =\big(\conj{z}\otimes Q_{B} u\mid \widetilde{\Bichar}\mid\conj{x}\otimes Q_{B}^{-1}y\big), 
  \end{equation}
  for all~\(x,z\in\Hils\), \(u\in\dom(Q_B)\) and~\(y\in\dom(Q_{B}^{-1})\). Here~\(Q_{B}\) is 
  the self\nb-adjoint operator defining manageability of~\(\Multunit[B]\in\U(\Hils[K]\otimes\Hils[K])\)
  in Definition~\ref{def:modularity}. Moreover, 
  \begin{equation}
    \label{eq:mang_bichar}
      \widetilde{\Bichar}\defeq \bichar^{\transpose\hat{\pi}\otimes\hat{\eta} 
      \Coinv_{\hat{B}}} \qquad\text{in \(\U(\conj{\Hils}\otimes\Hils[K])\).} 
  \end{equation} 
  Similarly, by duality,
   \(\DuBichar\defeq(\hat{\eta}\otimes\hat{\pi})\Dubichar\in\U(\Hils_B\otimes\Hils_A)\)
  is also manageable. 
  
 \begin{lemma}
  \label{lemm:V_tilde_commute_Q}
   \(\widetilde{\Bichar}\) and~\(Q_{A}^{\transpose}\otimes Q_{B}^{-1}\) commute.
 \end{lemma}
 \begin{proof}
  Combining Theorem~\ref{the:Cst_quantum_grp_and_mult_unit} \ref{eq:mang_rels} \ref{eq:sc_grp_Q} 
  and \cite{Meyer-Roy-Woronowicz:Homomorphisms}*{Proposition 3.10} we obtain
  \begin{equation}
   \label{eq:bichar_comm_QA_and_QB}
    (Q_{A}\otimes Q_{B})\Bichar (Q_{A}\otimes Q_{B})=\Bichar . 
  \end{equation}  
  Hence, in~\eqref{eq:second_leg_bichar_adapted}, we can replace \(x\), \(u\), \(z\) and~\(y\) by 
  \(Q^{\mathrm{i}t}_{A}(x)\), \(Q^{\mathrm{i}t}_{B}(u)\), \(Q^{\mathrm{i}t}_{A}(z)\) and~\(Q^{\mathrm{i}t}_{B}(y)\), 
  respectively, for all~\(t\in\R\). Thus we obtain 
  \[
     \big(\conj{z}\otimes Q_{B}u\mid\widetilde{\Bichar}\mid\conj{x}\otimes Q_{B}^{-1}y\big)
   = \Big(\big[Q_{A}^{\transpose}\big]^{-\mathrm{i} t}\conj{z}\otimes Q_{B}^{\mathrm{i}t}Q_{B}u\mid
     \widetilde{\Bichar}\mid \big[Q_{A}^{\transpose}\big]^{-\mathrm{i}t}\conj{x}\otimes
     Q_{B}^{\mathrm{i}t}Q^{-1}_{B}u\Big)  
  \]
  and therefore~\(\widetilde{\Bichar}=\Big(\big[Q_{A}^{\transpose}\big]^{\mathrm{i} t}\otimes Q^{-\mathrm{i} t}_{B}\Big)
 \widetilde{\Bichar}\Big(\big[Q_{A}^{\transpose}\big]^{\mathrm{i} t}\otimes Q^{-\mathrm{i} t}_{B}\Big)\) for all~\(t\in\R\).
 \end{proof}  
  By duality, \cite{Soltan-Woronowicz:Multiplicative_unitaries}*{Lemma 40} gives 
  \((\Coinv_{B}\otimes\Coinv_{\hat{B}})\Dumultunit[B]=\Dumultunit[B]\). 
  Using~\eqref{ex:exist_V-anti-Heisenberg_pair}, and 
  \eqref{eq:dual_right_morph_WB}  we compute 
  \begin{align*}
       \Dumultunit[B]_{\bar\beta 2}    
    =  \big((\hat{\Delta}_{R}\otimes\Id_{\hat{B}})\Dumultunit[B]\big)^{\transpose\hat{\eta}\otimes\transpose\hat{\pi}
        \otimes\Coinv_{\hat{B}}}
    &= \big(\bichar_{23}\Dumultunit[B]_{13}\big)^{\transpose\hat{\eta}\otimes\transpose\hat{\pi}
        \otimes\Coinv_{\hat{B}}}\\   
    &= (\Dumultunit[B])_{13}^{\transpose\hat{\eta}\otimes\Coinv_{\hat{B}}}
       \bichar_{23}^{\transpose\hat{\pi}\otimes\Coinv_{\hat{B}}}.  
  \end{align*}
  The third equality uses antimultiplicativity of~\(\transpose\hat{\eta}\otimes\transpose\hat{\pi}\otimes
  \Coinv_{\hat{B}}\). 
  
  Combining Theorem~\ref{the:Cst_quantum_grp_and_mult_unit}~\ref{eq:mang_rels}~\ref{eq:conj_mult},  
  \eqref{eq:gen_codob_reps}, and \eqref{eq:du_mult_H_s} we get
  \begin{equation}
   \label{eq:H_part_in_D_multunit}
     \Dumultunit[B]_{\theta\zeta}=\Bichar_{47}\DuMultunit[B]_{37}\widetilde{\DuMultunit[B]}^{*}_{17}\widetilde{\Bichar}
     ^{*}_{27} \quad\text{ in~\(\U(\conj{\Hils[K]}\otimes\conj{\Hils}\otimes\Hils[K]\otimes\Hils\otimes\conj{\Hils[K]}\otimes
     \conj{\Hils}\otimes\Hils[K]\otimes\Hils)\).}
  \end{equation}
  
  A similar computation yields
  \begin{equation}
   \label{eq:G_part_in_D_multunit}
   \Dumultunit[A]_{\rho\xi}=\DuMultunit[A]_{48}\widetilde{\DuMultunit[A]}^{*}_{28}
   \quad\text{ in~\(\U(\conj{\Hils[K]}\otimes\conj{\Hils}\otimes\Hils[K]\otimes\Hils\otimes\conj{\Hils[K]}\otimes
     \conj{\Hils}\otimes\Hils[K]\otimes\Hils)\).}
 \end{equation}
 
 \begin{theorem}
  \label{the:modularity_multunit_gen_drinfeld_double}
   The multiplicative unitary \(\DuMultunit[\mathcal{D}]\defeq\Dumultunit[B]_{\theta\zeta}\Dumultunit[A]_{\rho\xi}
   \in\U(\Hils_{\mathcal{D}}\otimes\Hils_{\mathcal{D}})\) is modular.
 \end{theorem}   
 \begin{proof}
  Recall~\(\Hils_{\mathcal{D}}=\conj{\Hils[K]}\otimes\conj{\Hils}\otimes\Hils[K]\otimes\Hils\). 
  Equations~\eqref{eq:H_part_in_D_multunit} and \eqref{eq:G_part_in_D_multunit} give 
  \begin{equation}
    \label{eq:Multunit_gen_codoub}
     \DuMultunit[\mathcal{D}]=\Bichar_{47}\DuMultunit[B]_{37}\widetilde{\Dumultunit[B]}^{*}_{17}\widetilde{\Bichar}
     ^{*}_{27}\DuMultunit[A]_{48}\widetilde{\DuMultunit[A]}^{*}_{28} 
     \in\U(\conj{\Hils[K]}\otimes\conj{\Hils}\otimes\Hils[K]\otimes\Hils\otimes\conj{\Hils[K]}\otimes
     \conj{\Hils}\otimes\Hils[K]\otimes\Hils).
   \end{equation}
  Define 
  \begin{equation}
   \label{eq:self_adj_op_mod_gen_drinfdoub}
     \hat{Q}_{\mathcal{D}}\defeq(Q^{-1}_{B})^{\transpose}\otimes (Q_{A}^{-1})^{\transpose}\otimes Q_{B}\otimes 
  Q_{A}\quad\text{ and }\quad Q_{\mathcal{D}}\defeq 1_{\conj{\Hils[K]}}\otimes 1_{\conj{\Hils}}\otimes Q_{B}\otimes Q_{A}. 
 \end{equation} 
  Clearly, \(\hat{Q}_{\mathcal{D}}\) and~\(Q_{\mathcal{D}}\) are positive, self-adjoint operators 
  on~\(\Hils_{\mathcal{D}}\) with trivial kernel.
 
 The commutation relations in Definition~\ref{def:modularity}~\ref{eq:modular_Q_hat_Q},  
 \cite{Woronowicz:Multiplicative_Unitaries_to_Quantum_grp}*{Proposition 1.4(1)} and 
 Lemma~\ref{lemm:V_tilde_commute_Q} show that~\(\hat{Q}_{\mathcal{D}}\otimes Q_{\mathcal{D}}\) 
 commutes with~\(\DuMultunit[B]_{\theta\zeta}\). A similar argument shows that 
 \(\hat{Q}_{\mathcal{D}}\otimes Q_{\mathcal{D}}\) commutes with \(\DuMultunit[A]_{\rho\xi}\). 
 Therefore, \(\hat{Q}_{\mathcal{D}}\otimes Q_{\mathcal{D}}\) commutes with 
 \(\DuMultunit[\mathcal{D}]\).
  
  Now the character condition~\eqref{eq:aux_W_char_in_first_leg} for~\(\Dumultunit[B]\) 
  and \cite{Woronowicz:Multiplicative_Unitaries_to_Quantum_grp}*{Theorem 1.7} show that 
  \(\Dumultunit[B]_{\theta\hat{\eta}}\in\U(\Hils_{\mathcal{D}}\otimes\Hils[K])\) is adapted 
  to~\(\DuMultunit[B]\in\U(\Hils[K]\otimes\Hils[K])\). Furthermore, 
  \cite{Woronowicz:Multiplicative_Unitaries_to_Quantum_grp}*{Theorem 1.6} shows that 
  \(\mathbb{X}\defeq((\Dumultunit[B])^{*}_{\theta2})
  ^{\transpose\otimes\hat{\eta}\Coinv_{\hat{B}}}\in\U(\conj{\Hils_{\mathcal{D}}}\otimes\Hils[K])\)
  satisfies a variant of~\eqref{eq:second_leg_bichar_adapted}. 
  Similarly, \(\mathbb{Y}\defeq((\Dumultunit[A])^{*}_{\rho2})^{\transpose\otimes\hat{\pi} 
  R_{\hat{A}}}\in\U(\conj{\Hils_{\mathcal{D}}}\otimes\Hils)\) is the unitary associated to the manageability of
  \(\Dumultunit[A]_{\rho\hat{\pi}}\in\U(\Comp(\Hils_{\mathcal{D}})\otimes\Hils)\).

  Clearly,~\(\xi\) and~\(\zeta\) act trivially on the factor~\(\conj{\Hils[K]}\otimes\conj{\Hils}\) of~\(\Hils_{\mathcal{D}}\).
  Therefore, following Notation~\ref{not:leg_numbering}, 
  we can write~\(\DuMultunit[\mathcal{D}]=
  \Dumultunit[B]_{\theta\hat{\eta}_{2}}\Dumultunit[A]_{\rho\hat{\pi}_{3}}\in\U(\Hils_{\mathcal{D}}\otimes\Hils[K]\otimes\Hils)\).

  Let~\(\{e_{i}\}_{i=1,2,\cdots}\) be an orthonormal basis of~\(\Hils_{\mathcal{D}}\).   
  The manageability condition~\eqref{eq:second_leg_bichar_adapted} 
  for~\(\mathbb{X}\) and~\(\mathbb{Y}\) gives
  \begin{align*}
  &   \big(x\otimes k\otimes h\mid\Dumultunit[B]_{\theta\hat{\eta}_{2}}\Dumultunit[A]_{\rho\hat{\pi}_{3}}\mid z\otimes k' \otimes h'\big)\\
  &= \Big(x\otimes k\otimes h\mid\Dumultunit[B]_{\theta\hat{\eta}_{2}}\sum_{i}\big(|e_{i})(e_{i}|\otimes 1_{\Hils[K]}\otimes 1_{\Hils}\big)
     \Dumultunit[A]_{\rho\hat{\pi}_{3}}\mid z\otimes k' \otimes h'\Big)\\
  &=\sum_{i}\Big(x\otimes k\mid\Dumultunit[B]_{\theta\hat{\eta}_{2}}\mid e_{i}\otimes k'\Big)
            \Big(e_{i}\otimes h\mid\Dumultunit[A]_{\rho\hat{\pi}_{3}}\mid z\otimes h'\Big)\\
  &=\sum_{i} \Big(\conj{z}\otimes Q_{A}(h)\mid\mathbb{Y}\mid\conj{e_{i}}\otimes Q^{-1}_{A}(h')\Big)
       \Big(\conj{e_{i}}\otimes Q_{B}(k)\mid \mathbb{X}\mid\conj{x}\otimes Q^{-1}_{B}(k')\Big)\\
  &= \Big(\conj{z}\otimes Q_{B}(k)\otimes Q_{A}(h)\mid \mathbb{Y}_{13} 
           \mathbb{X}_{12}\mid\conj{x}\otimes Q^{-1}_{B}(k')\otimes Q^{-1}_{A}(h')\Big), 
  \end{align*}
  where~\(x,z\in\Hils_{\mathcal{D}}\), \(k\in\dom(Q_B)\), \(h\in\dom(Q_{A})\), \(k'\in\dom(Q^{-1}_{B})\) 
  and~\(h'\in\dom(Q^{-1}_{A})\), respectively. 
  
  Hence \(\widetilde{\DuMultunit[\mathcal{D}]}\defeq \mathbb{Y}_{15}\mathbb{X}_{14}\in
  \U(\conj{\Hils_{\mathcal{D}}}\otimes\conj{\Hils[K]}\otimes\conj{\Hils}\otimes\Hils[K]\otimes\Hils)\) 
  satisfies condition~\ref{eq:modular_cond} in the Definition~\ref{def:modularity} 
  for~\(\DuMultunit[\mathcal{D}]\).
\end{proof} 
\section{Generalised quantum codoubles and generalised Drinfeld doubles}
  \label{sec:gen_qnt_codoub} 
 Let~\(\Qgrp{G}{A}\) and~\(\Qgrp{H}{B}\) be~\(\Cst\)\nb-quantum 
 groups and let~\(\bichar\in\U(\hat{A}\otimes\hat{B})\) be a bicharacter. 

 The map \(\flip^{\bichar}\colon\hat{B}\otimes\hat{A}\to\hat{A}\otimes\hat{B}\) defined by 
 \(\flip^{\bichar}(\hat{b}\otimes\hat{a})\defeq\bichar(\hat{a}\otimes\hat{b})\bichar^{*}\) 
 for~\(\hat{a}\in\hat{A}\), \(\hat{b}\in\hat{B}\) is an isomorphism of~\(\Cst\)\nb-algebras. 
 
 Define~\(\CodoubAlg_{\bichar}\defeq\hat{B}\otimes\hat{A}\) and 
 \(\DuComult[\DrinfDoubAlg_{\bichar}]\colon\CodoubAlg_{\bichar}
 \to\CodoubAlg_{\bichar}\otimes\CodoubAlg_{\bichar}\) by 
  \begin{equation}
    \label{eq:Gen_Co_doub_def} 
    \DuComult[\DrinfDoubAlg_{\bichar}](\hat{b}\otimes\hat{a})\defeq 
     (\Id_{\hat{B}}\otimes\flip^{\bichar}\otimes\Id_{\hat{A}})
     \big(\DuComult[B](\hat{b})\otimes\DuComult[A](\hat{a})\big).
 \end{equation} 
 
 Let~\((\rho,\theta)\) be the~\(\bichar\)\nb-Drinfeld pair on the 
 Hilbert space~\(\Hils_{\mathcal{D}}\) defined in~\eqref{eq:gen_codob_reps}. 
 Commutation relation~\eqref{eq:V-Drinfeld} gives a~\(\Cst\)\nb-algebra 
 \(\DrinfDoubAlg_{\bichar}\defeq\rho(A)\cdot\theta(B)\subset\Bound(\Hils_{\mathcal{D}})\). 
  
  Define~\(\Comult[\DrinfDoubAlg_{\bichar}]\colon
  \DrinfDoubAlg_{\bichar}\to\Mult(\DrinfDoubAlg_{\bichar}
  \otimes\DrinfDoubAlg_{\bichar})\) by 
  \begin{equation}
   \label{eq:Gen_Drinf_doub_def} 
     \Comult[\DrinfDoubAlg_{\bichar}](\rho(a)\theta(b))\defeq 
    (\rho\otimes\rho)\Comult[A](a)(\theta\otimes\theta)\Comult[B](b)
    \quad\text{  for all~\(a\in A\), \(b\in B\).}  
 \end{equation} 
 
This section is devoted to proving the following main result: 
\begin{theorem} 
 \label{the:Drinf_doub_codoub_dual} 
 Let~\(\DuMultunit[\mathcal{D}]\in\U(\Hils_{\mathcal{D}}\otimes\Hils_{\mathcal{D}})\) be 
 the modular multiplicative unitary in Theorem~\textup{\ref{the:modularity_multunit_gen_drinfeld_double}}. 
 Then
 \begin{enumerate} 
  \item\label{eq:V_codoub} \(\GenCodouble{\G}{\G[H]}{\bichar}\defeq(\CodoubAlg_{\bichar},
  \DuComult[\DrinfDoubAlg_{\bichar}])\) 
    is a~\(\Cst\)\nb-quantum group generated by~\(\DuMultunit[\mathcal{D}]\).
  \item\label{eq:V_Drinfdoub} \(\GenDrinfdouble{\G}{\G[H]}{\bichar}\defeq\Bialg{\DrinfDoubAlg_{\bichar}}\) 
  is the dual~\(\Cst\)\nb-quantum group of~\(\GenCodouble{\G}{\G[H]}{\bichar}\). 
  \item\label{eq:V-red_bichar} \(\DuMultunit[\mathcal{D}]\defeq\Dumultunit[B]_{\theta 2}\Dumultunit[A]_{\rho 3}
  \in\U(\DrinfDoubAlg_{\bichar}\otimes\hat{B}\otimes\hat{A})\) is the reduced bicharacter for \\
  \(\big(\GenCodouble{\G}{\G[H]}{\bichar},\GenDrinfdouble{\G}{\G[H]}{\bichar}\big)\). 
 \end{enumerate}
\end{theorem}  
\begin{definition}
  \label{def:Gen_codoub}
    The~\(\Cst\)\nb-quantum groups \(\GenCodouble{\G}{\G[H]}{\bichar}\) and 
    \(\GenDrinfdouble{\G}{\G[H]}{\bichar}\) are called the 
    \emph{generalised quantum codouble} and \emph{generalised Drinfeld double} 
    for the triple~\((\G,\G[H],\bichar)\), respectively.
 \end{definition}   
   \begin{remark}
  \label{rem:Rel_Baaj_Vaes}  
  Let~\(\DuG[H]^{\text{cop}}=(\hat{B},\flip\DuComult[B])\) be the coopposite 
  \(\Cst\)\nb-quantum group of~\(\DuG[H]\). According to the convention used 
  in~\cite{Baaj-Vaes:Double_cros_prod}*{Section 8}, the 
 map~\(m\colon\hat{B}\otimes\hat{A}\to\hat{B}\otimes\hat{A}\) 
 defined by~\(m(\hat{b}\otimes\hat{a})\defeq\Dubichar^{*} (\hat{b}\otimes\hat{a})\Dubichar\) 
 is an inner matching of~\(\DuG[H]^{\text{cop}}\) and \(\DuG\) 
 in sense of \cite{Baaj-Vaes:Double_cros_prod}*{Definition 3.1}. 
 In the presence of the Haar weights and the regularity assumption on~\(\G\) and~\(\G[H]\),  
 the \(\Cst\)\nb-algebraic version of the generalised quantum double of~\(\DuG[H]^{\text{cop}}\) 
 and~\(\DuG\) with respect~\(m\) in~\cite{Baaj-Vaes:Double_cros_prod} and  
 \(\GenCodouble{\G}{\G[H]}{\bichar}\) in Definition~\ref{def:Gen_codoub} are same.
\end{remark} 
 
 \begin{proposition}
  \label{prop:codoub_alg}
   Let~\(\DuMultunit[\mathcal{D}]\in\U(\Hils_{\mathcal{D}}\otimes\Hils_{\mathcal{D}})\) be 
   the modular multiplicative unitary in Theorem~\textup{\ref{the:modularity_multunit_gen_drinfeld_double}} 
   and \(\Multunit[\mathcal{D}]\in\U(\Hils_{\mathcal{D}}\otimes\Hils_{\mathcal{D}})\) be its dual. 
   Then 
   \begin{enumerate}
   \item\label{eq:Drinf_alg} \(\DrinfDoubAlg_{\bichar}=\{(\omega\otimes\Id_{\Hils_{\mathcal{D}}})
                     \Multunit[\mathcal{D}]:\omega
                     \in\Bound(\Hils_{\mathcal{D}})_{*}
                     \}^{\textup{CLS}}\). 
   \item\label{eq:Codoub_alg} \(\CodoubAlg_{\bichar}=\{(\omega\otimes\Id_{\Hils_{\mathcal{D}}})\DuMultunit[\mathcal{D}]:
                \omega\in\Bound(\Hils_{\mathcal{D}})_{*}\}^{\textup{CLS}}\).
  \end{enumerate}                             
 \end{proposition}
 \begin{proof}
  The representations~\(\xi\) and~\(\zeta\) in~\eqref{eq:gen_codob_reps} are faithful and 
  commute, hence~\(\Multunit[\mathcal{D}]=\multunit[A]_{2\rho}
  \multunit[B]_{1\theta}\in\U(\Comp(\Hils_{\mathcal{D}}) \otimes\hat{B}\otimes\hat{A})
  \).
   
  The  set of continuous linear functionals of the form~\(\eta\otimes\psi\) for~\(\eta\in 
  \hat{B}'\),  \(\psi\in\hat{A}'\) is linearly weak\(^*\) dense in~\((\hat{B}\otimes\hat{A})'\).
  Therefore, 
  \begin{align*}
  & \{(\omega\otimes\Id_{\Hils_{\mathcal{D}}})
      \Multunit[\mathcal{D}]:
  \omega\in\Bound(\Hils_{\mathcal{D}})_{*}\}^{\textup{CLS}}\\
 &= \{((\eta\otimes\psi\otimes\Id_{\Hils_{\mathcal{D}}})\multunit[A]_{2\rho} 
         \multunit[B]_{1\theta}) : 
  \psi\in\hat{A}',  \eta\in\hat{B}'\}^{\textup{CLS}} \\
 &= \{((\psi\otimes\Id_{\Hils_{\mathcal{D}}})\multunit[A]_{1\rho}) 
       ((\eta\otimes\Id_{\Hils_{\mathcal{D}}})\multunit[B]_{1\theta}) : 
  \psi\in\hat{A}',  \eta\in\hat{B}'\}^{\textup{CLS}}
  =\rho(A)\cdot\theta(B).
\end{align*} 

  Let~\(L=\{(\omega\otimes\Id_{\Hils_{\mathcal{D}}})\DuMultunit[\mathcal{D}]:\omega\in
  \Bound(\Hils_{\mathcal{D}})_{*}\}^{\text{CLS}}\).
  
 We identify~\(A\), \(\hat{A}\), \(B\), \(\hat{B}\) 
 with their images under the faithful representations~\(\pi\), \(\hat{\pi}\), \(\eta\), 
 \(\hat{\eta}\) to avoid complicated notation.
 
  Recall~\(\widetilde{\Bichar}^{*}=\bichar^{\transpose\Coinv_{\hat{A}}\otimes\Id_{\Hils[K]}}\), 
  \(\widetilde{\DuMultunit[A]}^{*}=(\Dumultunit[A])^{\transpose\Coinv_{A}\otimes\Id_{\Hils}}\),   
   and~\(\widetilde{\DuMultunit[B]}^{*}=(\Dumultunit[B])^{\transpose\Coinv_{B}\otimes\Id_{\Hils}}\).
    
  We rewrite~\eqref{eq:Multunit_gen_codoub} as
   \begin{equation*}
    \DuMultunit[\mathcal{D}]
    =\bichar_{47}\Dumultunit[B]_{37}
     (\Dumultunit[B])^{\transpose\Coinv_{B}\otimes\Id_{\Hils[K]}}_{17}
     \bichar^{\transpose\Coinv_{\hat{A}}\otimes\Id_{\Hils[K]}}_{27}
     \Dumultunit[A]_{48}
     (\Dumultunit[A])^{\transpose\Coinv_{A}\otimes\Id_{\Hils}}_{28}
   \end{equation*}   
   in~\(\U(\conj{\Hils[K]}\otimes\conj{\Hils}\otimes\Hils[K]\otimes\Hils\otimes
   \conj{\Hils[K]}\otimes\conj{\Hils}\otimes\Hils[K]\otimes\Hils)\).
   
 We replace \(\omega\in\Bound(\Hils_{\mathcal{D}})_{*}\)
 by \(\mu\otimes\epsilon\otimes\nu\otimes\upsilon\), where 
 \(\mu\in\Bound(\conj{\Hils[K]})_{*}\), \(\epsilon\in\Bound(\conj{\Hils})_{*}\), 
 \(\nu\in\Bound(\Hils[K])_{*}\), and~\(\upsilon\in\Bound(\Hils)_{*}\). 
 Next we use the leg numbering notation for functionals to denote 
 \(\mu\otimes\epsilon\otimes\nu\otimes\upsilon\otimes
 \Id_{\Hils_{\mathcal{D}}}\) by \(\mu_{1}\epsilon_{2}\nu_{3}\upsilon_{4}\). 
 Hence we have
 \[
  L=\left\{\mu_{1}\epsilon_{2}\nu_{3}\upsilon_{4}
           \Big(\bichar_{47}\Dumultunit[B]_{37}
     (\Dumultunit[B])^{\transpose\Coinv_{B}\otimes\Id_{\Hils[K]}}_{17}
     \bichar^{\transpose\Coinv_{\hat{A}}\otimes\Id_{\Hils[K]}}_{27}
     \Dumultunit[A]_{48}
     (\Dumultunit[A])^{\transpose\Coinv_{A}\otimes\Id_{\Hils}}_{28}\Big)
     \right\}^\text{CLS}.
 \]
 The slices of~\(\DuMultunit[B]\in\U(\Hils[K]\otimes\Hils[K])\) by functionals 
 \(\nu\in\Bound(\Hils[K])_{*}\) on the first leg generate a dense subspace of~\(B\).
 Therefore, we can replace~\((\nu\otimes\Id_{\Hils[K]})\Dumultunit[B]\) by~\(\hat{b}\in\hat{B}\)
 in the above expression and rewrite as follows:
  \[
  L=\left\{\mu_{1}\epsilon_{2}\upsilon_{3}
           \Big(\bichar_{36}\hat{b}_{6}
     (\Dumultunit[B])^{\transpose\Coinv_{B}\otimes\Id_{\Hils[K]}}_{16}
     \bichar^{\transpose\Coinv_{\hat{A}}\otimes\Id_{\Hils[K]}}_{26}
     \Dumultunit[A]_{37}
     (\Dumultunit[A])^{\transpose\Coinv_{A}\otimes\Id_{\Hils}}_{27}\Big)
     \right\}^{\text{CLS}}.
 \]
 Given~\(\mu\in\Bound(\conj{\Hils[K]})_{*}\) and~\(x\in\Bound(\conj{\Hils[K]})\), 
 define~\(x\cdot\mu(y)\defeq\mu(xy)\) for~\(y\in\Bound(\conj{\Hils[K]})\).

 Replacing~\(\mu\in\Bound(\conj{\Hils[K]})_{*}\) by~\(b^{\Coinv_{B}\transpose}\cdot\mu\), 
 for~\(b\in B\), \(L\)~becomes
   \[
     \left\{\mu_{1}\epsilon_{2}\upsilon_{3}
     \Big(\bichar_{36}\big((b\otimes\hat{b})
     (\Dumultunit[B])\big)^{\transpose\Coinv_{B}\otimes\Id_{\Hils[K]}}_{16}
     \bichar^{\transpose\Coinv_{\hat{A}}\otimes\Id_{\Hils[K]}}_{26}
     \Dumultunit[A]_{37}
     (\Dumultunit[A])^{\transpose\Coinv_{A}\otimes\Id_{\Hils}}_{27}\Big)
     \right\}^{\text{CLS}}.
 \]
 Since~\(\Dumultunit[B]\in\U(B\otimes\hat{B})\), we may replace~\((B\otimes\hat{B})\Dumultunit[B]\) 
 by~\(B\otimes\hat{B}\), and then applying~\(\mu\) on the first leg gives
 \[
    L=\left\{\epsilon_{1}\upsilon_{2}
     \Big(\bichar_{25}\hat{b}_{5}
     \bichar^{\transpose\Coinv_{\hat{A}}\otimes\Id_{\Hils[K]}}_{15}
     \Dumultunit[A]_{26}
     (\Dumultunit[A])^{\transpose\Coinv_{A}\otimes\Id_{\Hils}}_{16}\Big)
     \right\}^{\text{CLS}}.
 \]
 Replacing~\(\epsilon\in\Bound(\conj{\Hils})_{*}\) by~\(\hat{a}^{\transpose\Coinv_{\hat{A}}}\cdot\epsilon\) 
 for~\(\hat{a}\in\hat{A}\) yields
 \[
     L=\left\{\epsilon_{1}\upsilon_{2}
     \Big(\bichar_{25}\big((\hat{a}\otimes\hat{b})
     \bichar\big)^{\transpose\Coinv_{\hat{A}}\otimes\Id_{\Hils[K]}}_{15}
     \Dumultunit[A]_{26}
     (\Dumultunit[A])^{\transpose\Coinv_{A}\otimes\Id_{\Hils}}_{16}\Big)
     \right\}^{\text{CLS}}.
 \]
 Since~\(\bichar\in\U(\hat{A}\otimes\hat{B})\), we may replace~\((\hat{A}\otimes\hat{B})\bichar\) 
 by~\(\hat{A}\otimes\hat{B}\) in the last expression. Then we substitute   
 \(\hat{a}^{\transpose\Coinv_{\hat{A}}}\cdot\epsilon\) by~\(\epsilon\) and 
 the resulting expression becomes
  \[
     L=\left\{\epsilon_{1}\upsilon_{2}
     \Big(\bichar_{25}\hat{b}_{5}
     \Dumultunit[A]_{26}
     (\Dumultunit[A])^{\transpose\Coinv_{A}\otimes\Id_{\Hils}}_{16}\Big)
     \right\}^{\text{CLS}}.
 \]
 After replacing~\((\epsilon\otimes\Id_{\Hils})\Dumultunit[A]\) by~\(\hat{a}\in\hat{A}\)
 in the above expression, we obtain
 \[
     L=\left\{\upsilon_{1}
     \Big(\bichar_{14}\hat{b}_{4}
     \Dumultunit[A]_{15}\hat{a}_{5}\Big)
     \right\}^{\text{CLS}}.
 \]
 For all~\(\upsilon\in\Bound(\Hils)_{*}\) and~\(a\in A\subset\Bound(\Hils)\), 
 define~\(\upsilon\cdot a\in\Bound(\Hils)_{*}\) by~\(\upsilon\cdot a(y)\defeq\upsilon(ay)\) 
 for~\(y\in\Bound(\Hils)\). 
 
 Replacing~\(\upsilon\in\Bound(\Hils)_{*}\) by~\(\upsilon\cdot a\) 
 in the last expression gives
    \[
     L=\left\{\upsilon_{1}
     \Big(\bichar_{14}\hat{b}_{4}
     \big(\Dumultunit[A](a\otimes\hat{a}\big)_{15}\Big)
     \right\}^{\text{CLS}}
     =\left\{\upsilon_{1}
     (\bichar_{14}\hat{b}_{4}
    a_{1}\hat{a}_{5})
     \right\}^{\text{CLS}}
   =  \left\{\upsilon_{1}
     (\bichar_{14}\hat{b}_{4}
    \hat{a}_{5})
     \right\}^{\text{CLS}}.
 \]
 Finally, replacing~\(\upsilon\in\Bound(\Hils)_{*}\) by~\(\upsilon\cdot\hat{a}\) 
 for \(\hat{a}\in\hat{A}\) in the last expression, we get
     \[
    L=\left\{\upsilon_{1}\Big(\big(\Bichar(\hat{a}\otimes\hat{b})\big)_{14}\hat{a}_{5}\Big)\right\}^{\text{CLS}}
      =\left\{\upsilon_{1}\Big(\hat{a}_{1}\hat{b}_{4}\hat{a}_{5}\Big)\right\}^{\text{CLS}} 
      =\left\{\hat{b}_{3}\hat{a}_{4}\right\}^{\text{CLS}} 
      =1_{\conj{\Hils[K]}\otimes\conj{\Hils}}\otimes\hat{B}\otimes\hat{A}.\qedhere
     \] 
 \end{proof}
  So far it is not clear that~\(\Comult[\DrinfDoubAlg_{\bichar}]\) 
  is a well defined \(\Cst\)\nb-algebra morphism. For the moment, we assume it exists. 
 \begin{proposition}
   \label{prop:Cancellation_law}
    The comultiplication maps~\(\Comult[\DrinfDoubAlg_{\bichar}]\) 
    and~\(\DuComult[\DrinfDoubAlg_{\bichar}]\) defined by~\eqref{eq:Gen_Drinf_doub_def} 
    and~\eqref{eq:Gen_Co_doub_def} satisfy cancellation laws~\eqref{eq:Podles}. Equivalently, 
    \(\Bialg{\DrinfDoubAlg_{\bichar}}\) and~\((\CodoubAlg_{\bichar},\DuComult[\DrinfDoubAlg_{\bichar}])\) 
    are bisimplifiable \(\Cst\)\nb-bialgebras.
   \end{proposition}
   \begin{proof}
   A routine computation using coassociativity~\eqref{eq:coassociative} 
   and cancellation law \eqref{eq:Podles} for~\(\Comult[A]\) and~\(\Comult[B]\) 
   shows~\(\Bialg{\DrinfDoubAlg_{\bichar}}\) is a bisimplifiable 
   \(\Cst\)\nb-bialgebra.

   Cancellation law~\eqref{eq:Podles} for~\(\DuComult[A]\) gives
  \begin{align*}
   &   \DuComult[\DrinfDoubAlg_{\bichar}](\CodoubAlg_{\bichar})\cdot (1_{\CodoubAlg_{\bichar}}
         \otimes\CodoubAlg_{\bichar})\\
   &=  \bichar_{23}\big(\DuComult[B](\hat{B})_{13}\DuComult[A](\hat{A})_{24}\big)
          \bichar^{*}_{23}\cdot (1_{\hat{B}\otimes\hat{A}}\otimes
          \hat{B}\otimes\hat{A})\\
   &=  \bichar_{23}\Big(\DuComult[B](\hat{B})_{13}\big(\DuComult[A](\hat{A})\cdot (1_{\hat{A}}\otimes\hat{A})
         \big)_{24}\Big)\bichar^{*}_{23}   
         (1_{\hat{B}\otimes\hat{A}}\otimes\hat{B}\otimes 1_{\hat{A}})\\
  &=   \bichar_{23}\big(\DuComult[B](\hat{B})_{13}(\hat{A}\otimes\hat{A})_{24}\big)\bichar_{23}^{*}
        (1_{\hat{B}\otimes\hat{A}}\otimes\hat{B}\otimes 1_{\hat{A}}). 
 \end{align*}
 The character condition on the second leg~\eqref{eq:bichar_char_in_second_leg} 
 for~\(\bichar\) is equivalent to 
 \begin{equation}
  \label{eq:bichar_char_second_var_aux}
    \Dubichar_{\hat{\eta}'2}\Dumultunit[B]_{\eta'3}
   =\bichar_{23}\Dumultunit[B]_{\eta'3}\Dubichar_{\hat{\eta}'2}
   \qquad\text{in~\(\U(\Comp(\Hils_{\eta'})\otimes\hat{A}\otimes\hat{B})\),}
 \end{equation} 
 where~\((\eta,\hat{\eta}')\) is an~\(\G[H]\)\nb-Heisenberg pair acting on~\(\Hils_{\eta'}\).
 Recall that~\(\DuComult[B]\) is implemented by~\(\Dumultunit[B]\) 
 (see Equation~\eqref{eq:unit_imp_comult}). Therefore, we get 
 \begin{align*}
& \DuComult[\DrinfDoubAlg_{\bichar}](\CodoubAlg_{\bichar})_{\hat{\eta}'234}\cdot 
    (1_{\Hils_{\eta'}}\otimes 1_{\hat{A}}\otimes\CodoubAlg_{\bichar})\\
 &= \bichar_{23}\Dumultunit[B]_{\eta' 3}(\hat{\eta}'(\hat{B})\otimes\hat{A}\otimes  
 1_{\hat{B}}\otimes\hat{A})(\bichar_{23}\Dumultunit[B]_{\eta' 3})^{*}\cdot\hat{B}_{3}\\     
 &= \bichar_{23}\Dumultunit[B]_{\eta 3}(\hat{\eta}(\hat{B})\otimes\hat{A}\otimes 1_{\hat{B}}
 \otimes\hat{A})\Dubichar_{\hat{\eta} 2}(\Dumultunit[B]_{\eta 3})^{*}\Dubichar^{*}
 _{\hat{\eta} 2}\hat{B}_{3}. 
\end{align*} 
 Now~\(\bichar\in\U(\hat{A}\otimes\hat{B})\) gives 
 \((\hat{\eta}'(\hat{B})\otimes\hat{A})\Dubichar^{*}_{\hat{\eta}'2}=
 \hat{\eta}'(\hat{B})\otimes\hat{A}\) and~\(\bichar(\hat{A}\otimes\hat{B})=\hat{A}\otimes
 \hat{B}\). Using the cancellation law~\eqref{eq:Podles} for~\(\DuComult[B]\), 
 we obtain
 \begin{align*}
   \DuComult[\DrinfDoubAlg_{\bichar}](\CodoubAlg_{\bichar})_{\hat{\eta}'234}\cdot 
    (1_{\Hils_{\eta'}}\otimes 1_{\hat{A}}\otimes\CodoubAlg_{\bichar})
 &= \bichar_{23}\Dumultunit[B]_{\eta 3}(\hat{\eta}(\hat{B})\otimes\hat{A}\otimes 1_{\hat{B}}
       \otimes\hat{A})(\Dumultunit[B]_{\eta 3})^{*}\cdot\hat{B}_{3}\Dubichar^{*}_{\hat{\eta} 2}\\
  &=  \bichar_{23}\big((\DuComult[B](\hat{B})\cdot (1\otimes\hat{B})\big)_{\hat{\eta} 3}
         \Dubichar^{*}_{\hat{\eta} 2}\\         
  &= \bichar_{23}(\hat{\eta}(\hat{B})\otimes\hat{A}\otimes\hat{B}\otimes\hat{A}) 
        \Dubichar^{*}_{\hat{\eta} 2}\\
  &= \hat{\eta}(\hat{B})\otimes\hat{A}\otimes\hat{B}\otimes\hat{A}.
 \end{align*} 
 Since~\(\hat{\eta}\) is faithful, we get 
 \(\DuComult[\DrinfDoubAlg_{\bichar}](\CodoubAlg_{\bichar})\cdot 
 (1_{\CodoubAlg_{\bichar}}\otimes\CodoubAlg_{\bichar})=\CodoubAlg_{\bichar}\otimes\CodoubAlg_{\bichar}\). 
 A similar computation yields  \(\DuComult[\DrinfDoubAlg_{\bichar}](\CodoubAlg_{\bichar})\cdot 
 (\CodoubAlg_{\bichar}\otimes 1_{\CodoubAlg_{\bichar}})=\CodoubAlg_{\bichar}\otimes\CodoubAlg_{\bichar}\).
\end{proof}

 \begin{proof}[Proof of Theorem~\textup{\ref{the:Drinf_doub_codoub_dual}}]
  By virtue of Proposition~\ref{prop:codoub_alg} we can write 
  \(\DuMultunit[\mathcal{D}]=\Dumultunit[B]_{\theta2}\Dumultunit[A]_{\rho2}\in
  \U(\DrinfDoubAlg_{\bichar}\otimes\hat{B}\otimes A)\). Equivalently, 
  \(\Multunit[\mathcal{D}]=\multunit[A]_{2\rho}\multunit[B]_{1\theta}
  \in\U(\hat{B}\otimes\hat{A}\otimes\DrinfDoubAlg_{\bichar})\).
  
  The following computation takes place in~\(\U(\DrinfDoubAlg_{\bichar}\otimes\hat{B}\otimes\hat{A}\otimes
  \hat{B}\otimes\hat{A})\):
  \begin{align*}
     (\Id_{\DrinfDoubAlg_{\bichar}}\otimes\DuComult[\DrinfDoubAlg_{\bichar}])
       \DuMultunit[B]_{\theta2}\DuMultunit[A]_{\rho2}
  &=\bichar_{34}\Dumultunit[B]_{\theta2}\Dumultunit[B]_{\theta4}\Dumultunit[A]_{\rho3}
       \Dumultunit[A]_{\rho5}\bichar^{*}_{34}\\
  &= \Dumultunit[B]_{\theta2}\bichar_{34}\Dumultunit[B]_{\theta4}\Dumultunit[A]_{\rho3}          
        \bichar^{*}_{34}\Dumultunit[A]_{\rho5}\\
  &=  \Dumultunit[B]_{\theta2}\Dumultunit[A]_{\rho3}\Dumultunit[B]_{\theta4}\Dumultunit[A]_{\rho5}.  
 \end{align*}
 The first equality uses~\eqref{eq:Gen_Co_doub_def} and the character condition 
 \eqref{eq:aux_W_char_in_first_leg} for~\(\Dumultunit[A]\) and~\(\Dumultunit[B]\), 
 the second equality uses that~\(\bichar\) with \(\Dumultunit[B]_{\theta2}\) and 
 \(\Dumultunit[A]_{\rho5}\), and the last equality uses~\eqref{eq:rho_theta_equiv_Drinfeld_cond}. 
 
 Collapsing the leg numbers we obtain~\eqref{eq:W_char_in_second_leg} for 
 \(\DuComult[\DrinfDoubAlg_{\bichar}]\) and~\(\DuMultunit[\mathcal{D}]\):
 \begin{equation}
  \label{eq:codoub_bichar_1}
       (\Id_{\DrinfDoubAlg_{\bichar}}\otimes\DuComult[\DrinfDoubAlg_{\bichar}])\DuMultunit[\mathcal{D}]
    = \DuMultunit[\mathcal{D}]_{12}\DuMultunit[\mathcal{D}]_{13} 
    \qquad\text{in~\(\U(\DrinfDoubAlg_{\bichar}\otimes\CodoubAlg_{\bichar}\otimes\CodoubAlg_{\bichar})\).}
\end{equation}  
 Combining~\eqref{eq:codoub_bichar_1} with Proposition~\ref{prop:Cancellation_law} gives~\ref{eq:V_codoub}.

  Next we establish~\ref{eq:V_Drinfdoub}. The character condition on the second 
  leg~\eqref{eq:W_char_in_second_leg} for \(\multunit[A]\) and \(\multunit[B]\) yields
  \begin{align*}
   & \big(\Id_{\hat{B}\otimes\hat{A}}\otimes\Comult[\DrinfDoubAlg_{\bichar}]\big)
      \multunit[A]_{2\rho}\multunit[B]_{1\theta}\\
   &= \Big(\big(\Id_{\hat{A}}\otimes(\rho\otimes\rho)\Comult[A])\big)
   \multunit[A]\Big)_{234}\Big(\big(\Id_{\hat{B}}\otimes(\theta\otimes\theta)\Comult[B])
   \big)\Dumultunit[B]\Big)_{134}\\
   &= \multunit[A]_{2\rho_{3}}\multunit[A]_{2\rho_{4}}\multunit[B]_{1\theta_{3}}
        \multunit[B]_{1\theta_{4}}\\
   &=  \multunit[A]_{2\rho_{3}}\multunit[B]_{1\theta_{3}}\multunit[A]_{2\rho_{4}}
        \multunit[B]_{1\theta_{4}}
   \qquad\text{in~\(\U(\hat{B}\otimes\hat{A}\otimes\DrinfDoubAlg_{\bichar}\otimes\DrinfDoubAlg_{\bichar})\).}     
  \end{align*}
 Here we use Notation~\ref{not:leg_numbering} for the representations~\(\rho\) 
 and~\(\theta\).  
 Collapsing the first two legs we obtain~\eqref{eq:W_char_in_second_leg} for 
 \(\Comult[\DrinfDoubAlg_{\bichar}]\) and~\(\Multunit[\mathcal{D}]\):
 \begin{equation}
  \label{eq:codoub_bichar_2}
      (\Id\otimes\Comult[\DrinfDoubAlg_{\bichar}])\Multunit[\mathcal{D}] 
     =  \Multunit[\mathcal{D}]_{12}\Multunit[\mathcal{D}]_{13}
     \qquad\text{in~\(\U(\CodoubAlg_{\bichar}\otimes\DrinfDoubAlg_{\bichar}\otimes\DrinfDoubAlg_{\bichar})\).}       
 \end{equation}
  Hence, Theorem~\ref{the:Cst_quantum_grp_and_mult_unit} ensures the 
  existence and uniqueness of \(\Comult[\DrinfDoubAlg_{\bichar}]\) as an element in 
  \(\Mor(\DrinfDoubAlg_{\bichar},\DrinfDoubAlg_{\bichar}\otimes\DrinfDoubAlg_{\bichar})\) 
  satisfying~\eqref{eq:codoub_bichar_2}, and  
  \(\GenDrinfdouble{\G}{\G[H]}{\bichar}\) is a~\(\Cst\)\nb-quantum group 
 generated by the modular multiplicative unitary~\(\Multunit[\mathcal{D}]\); 
 hence it is dual of \(\GenCodouble{\G}{\G[H]}{\bichar}\).
 
 From~\eqref{eq:codoub_bichar_1} and~\eqref{eq:codoub_bichar_2}, it is 
 clear that~\(\DuMultunit[\mathcal{D}]\defeq
 \Dumultunit[B]_{\theta2}\Dumultunit[A]_{\rho3}\in\U(\DrinfDoubAlg_{\bichar}
 \otimes\CodoubAlg_{\bichar})\) is a bicharacter, and its Hilbert space 
 realisation is a modular multiplicative unitary generating 
 \(\GenCodouble{\G}{\G[H]}{\bichar}\). 
 
 Now the representations~\(\rho\) and~\(\theta\) defined in~\eqref{eq:gen_codob_reps}  
 depends on the \(\G\)\nb-Heisenberg pair~\(\HeisPair{\pi}\) and~\(\G[H]\)\nb-Heisenberg 
 pair~\(\HeisPair{\eta}\). Hence, on one hand, the \(\Cst\)\nb-algebra \(\DrinfDoubAlg_{\bichar}\) 
 depends on the representations~\(\pi\), \(\hat{\pi}\), \(\eta\) and~\(\hat{\eta}\).
 
 On the other hand, \((\CodoubAlg_{\bichar},\DuComult[\DrinfDoubAlg_{\bichar}])\) only 
 depends on the triple~\((\G,\G[H],\bichar)\). Therefore, by virtue of 
 Theorem~\ref{the:ind_multunit}, \(\GenDrinfdouble{\G}{\G[H]}{\bichar}\) does not 
 depend on the choice of~\(\Multunit[\mathcal{D}]\), which in turn, shows that 
 \(\DrinfDoubAlg_{\bichar}\) does not depend on \(\pi\), \(\hat{\pi}\), \(\eta\) 
 and~\(\hat{\eta}\). Hence, \(\DuMultunit[\mathcal{D}]\) is the reduced 
 bicharacter for \((\GenCodouble{\G}{\G[H]}{\bichar},\GenDrinfdouble{\G}{\G[H]}{\bichar})\).
\end{proof}
 \begin{remark} 
  \label{rem:ind_drinf_doub}
  By definition of the generalised quantum codouble 
  \eqref{eq:Gen_Co_doub_def}, the pair 
  \(\Bialg{\CodoubAlg_{\bichar}}\) only depends on the 
  triple~\((\G,\G[H],\bichar)\). Also, Theorem~\ref{the:ind_multunit} 
  ensures that the generalised Drinfeld double 
  \(\GenDrinfdouble{\G}{\G[H]}{\bichar}\) is uniquely determined 
  (up to isomorphism) by its dual \(\Bialg{\CodoubAlg_{\bichar}}\). 
  Hence, the generalised Drinfeld double also depends only on the 
  triple~\((\G,\G[H],\bichar)\).
 \end{remark} 
  \begin{definition} 
    \label{def:can-Drinf_pair} 
     The pair \((\rho,\theta)\) in~\eqref{eq:gen_codob_reps} 
     is called a \emph{canonical \(\bichar\)\nb-Drinfeld pair.}
  \end{definition}  
 Next we gather other structure maps on the generalised quantum 
 codouble.  
 \begin{proposition}
 \label{prop:scaling_grp_unit_anitpode_gen_codoub}
  Let~\(\Bialg{\CodoubAlg_{\bichar}}\) be the generalised quantum codouble for the 
  triple~\((\G,\G[H],\bichar)\). Then
  \begin{enumerate}
    \item\label{eq:coinv_codoub} \(\Coinv_{\CodoubAlg_{\bichar}}(\hat{b}\otimes\hat{a})\defeq\Dubichar(\Coinv_{\hat{B}}(\hat{b})\otimes\Coinv_{\hat{A}}(\hat{a}))\Dubichar^{*}\) 
            is the unitary antipode, 
    \item\label{eq:sc_codoub} \(\tau^{\CodoubAlg_{\bichar}}_{t}(\hat{b}\otimes\hat{a})\defeq\tau^{\hat{B}}_{t}(\hat{b})\otimes\tau^{\hat{A}}_{t}(\hat{a})\) for~\(t\in\R\)
           is the scaling group,          
  \end{enumerate}
      of~\(\GenCodouble{\G}{\G[H]}{\bichar}\), where~\(\hat{a}\in\hat{A}\), \(\hat{b}\in\hat{B}\).        
 \end{proposition}
 \begin{proof} 
  To conclude~\ref{eq:coinv_codoub} it is sufficient to show 
  Theorem~\ref{the:Cst_quantum_grp_and_mult_unit}~\ref{eq:mang_rels}~\ref{eq:conj_mult} for 
  \(\Coinv_{\CodoubAlg_{\bichar}}\). Let~\(\HeisPair{\pi}\) and~\(\HeisPair{\eta}\) be~\(\G\) 
  and~\(\G[H]\)\nb-Heisenberg pairs acting on~\(\Hils\) and \(\Hils[K]\), respectively. 
   The proof of Theorem~\ref{the:modularity_multunit_gen_drinfeld_double} shows that 
   \[
     \widetilde{\DuMultunit[\mathcal{D}]}^{*}=(\Dumultunit[B]_{\theta 2})^{\transpose\otimes\hat{\eta}\Coinv_{\hat{B}}} 
   (\Dumultunit[A]_{\rho 3})^{\transpose\otimes\hat{\pi}\Coinv_{\hat{A}}}
   \qquad\text{in~\(\U(\conj{\Hils_{\mathcal{D}}}\otimes\Hils[K]\otimes\Hils)\).}
   \] 
   We rewrite~\eqref{eq:rho_theta_equiv_Drinfeld_cond} in the following way:
   \begin{equation}
    \label{eq:Exist_Gen_codoub_compatible_corep}
         \Dubichar_{23}\Dumultunit[A]_{\rho 3}\Dumultunit[B]_{\theta 2}
     =  \Dumultunit[B]_{\theta 2}\Dumultunit[A]_{\rho 3}\Dubichar_{23}
     \qquad\text{in~\(\U(\Comp(\Hils_{\mathcal{D}})\otimes\hat{B}\otimes\hat{A})\).}
   \end{equation}
  \cite{Meyer-Roy-Woronowicz:Homomorphisms}*{Proposition 3.10} gives 
  \((\Coinv_{\hat{B}}\otimes\Coinv_{\hat{A}})\Dubichar=\Dubichar\). Next, applying the 
  antimultiplicative map~\(\transpose\otimes\Coinv_{\hat{B}}\otimes\Coinv_{\hat{A}}\) 
  to the both sides of~\eqref{eq:Exist_Gen_codoub_compatible_corep}, gives
  \[
      (\Dumultunit[B]_{\theta 2})^{\transpose\otimes\Coinv_{\hat{B}}}(\Dumultunit[A]_{\rho 3})^{\transpose\otimes\Coinv_{\hat{A}}}\Dubichar_{23}
    =\Dubichar_{23}(\Dumultunit[A]_{\rho 3})^{\transpose\otimes\Coinv_{\hat{A}}}(\Dumultunit[B]_{\theta 2})^{\transpose\otimes\Coinv_{\hat{B}}}
    \quad\text{in~\(\U(\Comp(\conj{\Hils_{\mathcal{D}}})\otimes\hat{B}\otimes\hat{A})\).}
  \] 
  Combining the  first and last equations above, we get
  \[
       \widetilde{\DuMultunit[\mathcal{D}]}^{*}
    = \DuBichar_{23}(\Dumultunit[A]_{\rho 3})^{\transpose\otimes\hat{\pi}\Coinv_{\hat{A}}}(\Dumultunit[B]_{\theta 2})^{\transpose\otimes\hat{\eta}\Coinv_{\hat{B}}}
       \DuBichar_{23}^{*}
    = \DuBichar_{23}(\Dumultunit[B]_{\theta 2}\Dumultunit[A]_{\rho 3})^{\transpose\otimes\hat{\eta}\Coinv_{\hat{B}}\otimes\hat{\pi}\Coinv_{\hat{A}}}
       \DuBichar_{23}^{*} 
 \]       
 in~\(\U(\conj{\Hils_{\mathcal{D}}}\otimes\Hils[K]\otimes\Hils)\) and~\(\DuBichar\defeq(\hat{\eta}\otimes\hat{\pi})\Dubichar\in\U(\Hils[K]\otimes\Hils)\); 
 hence \(\widetilde{\DuMultunit[\mathcal{D}]}^{*}=\DuMultunit[\mathcal{D}]{\xspace}^{\transpose\otimes\Coinv_{\CodoubAlg_{\bichar}}}\). 
 
 Recall the positive self\nb-adjoint operator~\(Q_{\mathcal{D}}=1_{\conj{\Hils[K]}\otimes\conj{\Hils}}\otimes Q_{B}\otimes Q_{A}\) 
 from~\eqref{eq:self_adj_op_mod_gen_drinfdoub} on~\(\Hils_{\mathcal{D}}\). 
  Theorem~\ref{the:Cst_quantum_grp_and_mult_unit}~\ref{eq:mang_rels}~\ref{eq:sc_grp_Q} gives 
 \[
     Q_{\mathcal{D}}^{2\mathrm{i}t}(\xi(\hat{a})\zeta(\hat{b}))Q_{\mathcal{D}}^{-2\mathrm{i}t}
  = Q_{\mathcal{D}}^{2\mathrm{i}t}(1_{\conj{\Hils[K]}\otimes\conj{\Hils}}\otimes\hat{\eta}(\hat{b})\otimes\hat{\pi}(\hat{a}))Q_{\mathcal{D}}^{-2\mathrm{i}t}
  = 1_{\conj{\Hils[K]}\otimes\conj{\Hils}}\otimes\hat{\eta}(\tau_{t}^{\hat{B}}(\hat{b}))\otimes\hat{\pi}(\tau_{t}^{\hat{A}}(\hat{a}))
 \]
  for all~\(\hat{a}\in\hat{A}\), \(\hat{b}\in\hat{B}\). 
  Finally, faithfulness of \(\hat{\pi}\) and~\(\hat{\eta}\) gives~\ref{eq:sc_codoub}.
 \end{proof}
 Similarly, we can prove the following result.
   \begin{proposition}
    \label{prop:scaling_grp_unit_anitpode_gen_Drinfdoub}
      Let \(\Bialg{\DrinfDoubAlg_{\bichar}}\) be the generalised Drinfeld 
      double for the triple~\((\G,\G[H],\bichar)\). Then 
     \begin{enumerate}
        \item the map~\(\Coinv_{\DrinfDoubAlg_{\bichar}}(\rho(a)\theta(b))\defeq\rho(\Coinv_{A}(a))\theta(\Coinv_{B}(b))\) 
                 defines the unitary antipode,  
        \item  \(\left\{\tau_{t}^{\DrinfDoubAlg_{\bichar}}(\rho(a)\theta(b))
                    \defeq\rho(\tau^{A}_{t}(a))\theta(\tau^{B}_{t} (b))\right\}_{t\in\R}\) 
                     is the scaling group, 
     \end{enumerate}
      on~\(\GenDrinfdouble{\G}{\G[H]}{\bichar}\) for all~\(a\in A\) and~\(b\in B\).
 \end{proposition}
 
  \begin{example}
  \label{ex:triv_bichar}
   Let~\(\bichar=1_{\hat{A}}\otimes 1_{\hat{B}}\in\U(\hat{A}\otimes\hat{B})\).
   Then~\(\rho\) and~\(\theta\) in~\eqref{eq:V-Drinfeld} commute. Then we 
   identify \(\Hils_{\mathcal{D}}\) with~\(\Hils[K]\otimes\Hils\), and  
   \(\Multunit[\mathcal{D}]\) with~\(\Multunit[A]_{24}\Multunit[B]_{13}\);  
   hence~\(\GenDrinfdouble{\G}{\G[H]}{\bichar}\) becomes the product 
   of~\(\G\) and~\(\G[H]\), denoted by~\(\G\times\G[H]\). 
   Equivalently, \(\DrinfDoubAlg_{\bichar}=B\otimes A\) and 
   \(\Comult[\DrinfDoubAlg_{\bichar}](\hat{b}\otimes\hat{a})
   =\DuComult[B](\hat{b})_{13}\DuComult[A](\hat{a})_{24}\) for~\(\hat{a}\in\hat{A}\), 
   \(\hat{b}\in\hat{B}\). Similarly, \(\GenCodouble{\G}{\G[H]}{\bichar}\) 
   becomes the product of~\(\DuG\) and~\(\DuG[H]\).
 \end{example}
 
 \begin{example}
  \label{ex:comm_case_Gen_drinf_doub}
   Let~\(\hat{A}=\Contvin(G)\) and~\(\hat{B}=\Contvin(H)\) for 
   locally compact groups~\(G\) and~\(H\), respectively. 
   For any bicharacter~\(\bichar\in\U(\hat{A}\otimes\hat{B})\),  
   the representations~\(\rho\) and~\(\theta\) satisfying~\eqref{eq:V-Drinfeld} 
   commute. By Example~\ref{ex:triv_bichar}, we identify
   \(\Hils_{\mathcal{D}}=L^{2}(H\times G)\) 
   with respect to the right Haar measures on~\(G\) and~\(H\).   
   The the multiplicative unitary \(\Multunit[\mathcal{D}]\defeq\Multunit[A]_{24}\Multunit[B]_{13}\) 
   is defined by \(\Multunit[A]_{24}\Multunit[B]_{13}f(h_{1},g_{1},h_{2},g_{2})\defeq 
   f(h_{1}h_{2},h_{2},g_{1}g_{2},g_{2})\) for 
   \(f\in L^{2}(H\times G\times H\times G)\) and~\(g_{1},g_{2}\in G\), 
   \(h_{1},h_{2}\in H\). Then~\(\CodoubAlg_{\bichar}=\Contvin(H\times G)\) and 
   \(\DrinfDoubAlg_{\bichar}=\Cred(H\times G)\).
 \end{example}
 
 \begin{example}
  \label{ex:Gen_Drinf_doub_qnt_grp}
  In particular, let~\(\hat{B}=A\), \(\DuComult[B]=\Comult[A]\), and~\(\bichar=\multunit[A]\in 
  \U(\hat{A}\otimes A)\). Let~\((\pi,\hat{\pi})\) be a~\(\G\)\nb-Heisenberg pair
  on a Hilbert space~\(\Hils\) and let~\((\bar\pi,\bar{\hat{\pi}})\) be the 
  corresponding~\(\G\)\nb-anti-Heisenberg pair on~\(\conj{\Hils}\).  
  We can simplify~\eqref{eq:gen_codob_reps} as follows: 
  \(\Hils_{\mathcal{D}}=\conj{\Hils}\otimes\Hils\otimes\Hils\), 
  \(\rho(a)\defeq(\bar{\pi}\otimes\pi)\Comult[A](a)_{13}\), \(\theta(\hat{a})\defeq
  ((\bar{\hat{\pi}}\otimes\hat{\pi})\DuComult[A]\otimes\hat{\pi})\DuComult[A](\hat{a})\), 
  \(\xi(\hat{a})\defeq\Id_{\conj{\Hils}\otimes\Hils}\otimes\hat{\pi}(\hat{a})\), 
  \(\zeta(a)\defeq\Id_{\conj{\Hils}}\otimes\pi(a)\otimes\Id_{\Hils}\),
   for~\(a\in A\), \(\hat{a}\in\hat{A}\), 
   respectively. Then the~\(\multunit[A]\)\nb-Drinfeld double 
  is called the~\(\G\)\nb-\emph{Drinfeld double} and denoted by   
  \(\Drinfdouble{\G}=(\DrinfDoubAlg^{A},\Comult[\DrinfDoubAlg^{A}])\). 
  Here~\(\DrinfDoubAlg^{A}\defeq\rho(A)\cdot\theta(\hat{A})\), and 
  \(\Comult[\DrinfDoubAlg^{A}](\rho(a)\cdot\theta(\hat{a})
  \defeq(\rho\otimes\rho)\Comult[A](a)\cdot (\theta\otimes\theta)\DuComult[A](\hat{a})\) 
  for \(a\in A\), \(\hat{a}\in\hat{A}\). Similarly, the dual of~\(\Drinfdouble{\G}\) is called the 
  \(\G\)\nb-\emph{quantum codouble} and denoted by 
  \(\Codouble{\G}=(\CodoubAlg^{A},\Comult[\CodoubAlg^{A}])\). Here 
  \(\CodoubAlg^{A}\defeq A\otimes\hat{A}\), and 
  \(\Comult[\CodoubAlg^{A}](a\otimes\hat{a})\defeq \flip^{\multunit[A]}_{23}\big(\Comult[A](a)\otimes
  \DuComult[A](\hat{a})\big)\) for~\(a\in A\), \(\hat{a}\in\hat{A}\).
 \end{example} 
 \begin{example}
  \label{ex:Drinf_doub_group_case} 
  Let~\(A=\Contvin(G)\) and~\(\hat{A}=\Cred(G)\) for  
  a locally compact group~\(G\). Then 
  \cite{Kahng:Twisting_Qnt_doub}*{Proposition 5.1} shows that 
  the underlying \(\Cst\)\nb-algebra of the Drinfeld double of~\(G\), 
  denoted by~\(\DrinfDoubAlg^{\Contvin(G)}\), is \(\Contvin(G)\rtimes G\) for 
  the conjugation action of~\(G\) on itself. 
 \end{example} 
  
 \section{Properties of generalised Drinfeld doubles}
  \label{sec:prop_Gen_Drinf}
  We start with the noncommutative version of the following 
  classical fact: given two locally compact groups~\(G\) and~\(H\), 
  there are canonical Hopf~\(^*\)\nb-homomorphisms from~\(\Contvin(G)\) 
  and~\(\Contvin(H)\) to~\(\Contvin(G\times H)\). 
  \label{subsec:Prop_gen_drin_doub}
 \begin{lemma}
  \label{lemm:embedding_to_Gen_Drinf}
   The unitaries~\(\multunit[A]_{1\rho}\in\U(\hat{A}\otimes\DrinfDoubAlg_{\bichar})\) 
   and~\(\multunit[B]_{1\theta}\in\U(\hat{B}\otimes\DrinfDoubAlg_{\bichar})\) 
   are bicharacters induced by the Hopf~\(^*\)\nb-homomorphisms~\(\rho\in\Mor(A,
   \DrinfDoubAlg_{\bichar})\) and~\(\theta\in\Mor(B,\DrinfDoubAlg_{\bichar})\), respectively.
 \end{lemma}
  \begin{proof}
   The character condition on the first leg~\eqref{eq:bichar_char_in_first_leg},
   for both the unitaries, follows from~\eqref{eq:W_char_in_first_leg}. 
  Now equation~\eqref{eq:unit_imp_comult} for~\(\Comult[\DrinfDoubAlg_{\bichar}]\) 
  and \(\Multunit[\mathcal{D}]\) yields
 \begin{equation}
   \label{eq:Explicit_comult_of_Gen_Drinf_doub}
      \Comult[\DrinfDoubAlg_{\bichar}](\rho(a)\cdot\theta(b))
   =   (\Multunit[\mathcal{D}])(\rho(a)\cdot\theta(b)\otimes 
        1)(\Multunit[\mathcal{D}])^{*} 
    \qquad\text{for all~\(a\in A\), \(b\in B\).}
  \end{equation}        
   Using~\eqref{eq:Explicit_comult_of_Gen_Drinf_doub} we write 
   \[
       (\Id_{\hat{A}}\otimes\Comult[\DrinfDoubAlg_{\bichar}])\multunit[A]_{1\rho}
    = \multunit[A]_{\xi_{2}\rho_{3}}\multunit[B]_{\zeta_{2}\theta_{3}}
        \multunit[A]_{1\rho_{2}}(\multunit[B]_{\zeta_{2}\theta_{3}})^{*}
       (\multunit[A]_{\xi_{2}\rho_{3}})^{*}
   \quad\text{in~\(\U(\hat{A}\otimes\DrinfDoubAlg^{\bichar}\otimes\DrinfDoubAlg^{\bichar})
   \).}
  \]  
  By Lemma~\ref{lemm:commutation_bet_diff_reps}, 
  \(\zeta\) and~\(\rho\) commute and~\((\rho,\xi)\) is a 
  \(\G\)\nb-Heisenberg pair. This yields~\eqref{eq:bichar_char_in_second_leg} 
  for~\(\multunit[A]_{1\rho}\):
  \[
     (\Id_{\hat{A}}\otimes\Comult[\DrinfDoubAlg_{\bichar}])\multunit[A]_{1\rho}
   =   \multunit[A]_{\xi_{2}\rho_{3}}\multunit[A]_{1\rho_{2}}
        (\multunit[A]_{\xi_{2}\rho_{3}})^{*}
  = \multunit[A]_{1\rho_{2}}\multunit[A]_{1\rho_{3}}
     \quad\text{in~\(\U(\hat{A}\otimes\DrinfDoubAlg^{\bichar}\otimes\DrinfDoubAlg^{\bichar})
     \).}
  \]   
  Furthermore, taking slices on the first leg of the last expression by~\(\omega\in\hat{A}'\) 
  and using~\eqref{eq:W_char_in_second_leg} for~\(\multunit[A]\) we get
  \(\Comult[\DrinfDoubAlg_{\bichar}]\rho(a)=(\rho\otimes\rho)\Comult[A](a)\) 
  for~\(a\in A\). Therefore,~\(\rho\) is a Hopf~\(^*\)\nb-homomorphism from 
  \(\G\) to~\(\GenDrinfdouble{\G}{\G[H]}{\bichar}\) and \(\multunit[A]_{1\rho}\) 
  is induced by~\(\rho\).   
  
  Similarly, we can show that~\(\multunit[B]_{1\theta}\) is induced 
  by the Hopf~\(^*\)\nb-homomorphism from~\(\G[H]\) to 
  \(\GenDrinfdouble{\G}{\G[H]}{\bichar}\). 
\end{proof}   
\subsection{Coaction on the twisted tensor product of \texorpdfstring{$\Cst$}{C*}-algebras}
 \label{subsec:Act_on_Drinf_doub} 
 \(\Cst\)\nb-algebras can be turned into a category, which we 
 generically denote by~\(\Cstcat\), using several types of maps:
 \begin{itemize}
\item morphisms (nondegenerate \Star{}homomorphisms \(C_1\to
  \Mult(C_2)\));
\item proper morphisms (nondegenerate \Star{}homomorphisms \(C_1\to
  C_2\));
\item completely positive maps \(C_1\to C_2\);
\item completely positive contractions \(C_1\to C_2\);
\item completely contractive maps \(C_1\to C_2\);
\item completely bounded maps \(C_1\to C_2\).
\end{itemize}
 Let \(\Cstcat(\G)\) generically denote the category with 
 \(\G\)\nb-\(\Cst\)-algebras as objects and \(\G\)\nb-equivariant
 ``maps'' as arrows. 
 
The twisted tensor product construction in~\cite{Meyer-Roy-Woronowicz:Twisted_tensor} 
of a~\(\G\)\nb- and an~\(\G[H]\)\nb-\(\Cst\)\nb-algebra 
with respect to~\(\bichar\), denoted by~\(\boxtimes_{\bichar}\), defines 
a bifunctor from \(\Cstcat(\G)\times\Cstcat(\G[H])\) to~\(\Cstcat\) 
(see \cite{Meyer-Roy-Woronowicz:Twisted_tensor}*{Lemma 5.5}).

 In particular, if \(\bichar=1_{\hat{A}}\otimes 1_{\hat{B}}\), then 
 \(C\boxtimes_{\bichar}D=C\otimes D\), and~\(\GenDrinfdouble{\bichar}{\G}{\G[H]}\) 
 becomes the product of~\(\G\) and~\(\G[H]\) (see Example~\ref{ex:triv_bichar}). 
 Then the map \(c\otimes d\mapsto \gamma(c)_{13}\delta(d)_{24}\) 
 defines the coaction of the product of~\(\G\) and~\(\G[H]\) on \(C\otimes D\). 
 Equivalently, \(\otimes\colon\Cstcat(\G)\times\Cstcat(\G[H])\to\Cstcat(\G\times\G[H])\) 
 is a bifunctor. The following theorem is a noncommutative version of this fact. 
\begin{theorem}
 \label{the:bifunctor}
  \(\boxtimes_{\bichar}\colon\Cstcat(\G)\times\Cstcat(\G[H])\to
   \Cstcat(\GenDrinfdouble{\G}{\G[H]}{\bichar})\) 
  is a bifunctor.  
\end{theorem} 
Let~\((C,\gamma)\) and \((D,\delta)\) be \(\G\)\nb- and 
\(\G[H]\)\nb-\(\Cst\)\nb-algebras, respectively. 
Let~\((\alpha,\beta)\) be a~\(\bichar\)\nb-Heisenberg pair on a Hilbert space~\(\Hils[L]\).  
Then~\(C\boxtimes_{\bichar}D\), is defined by~\(C\boxtimes_{\bichar}D\defeq
\iota_{C}(C)\cdot\iota_{D}(D)\subset\Mult(C\otimes D\otimes\Comp(\Hils[L]))\)
 (see~\cite{Meyer-Roy-Woronowicz:Twisted_tensor}*{Lemma 3.11}). 
 Here~\(\iota_{C}(c)\defeq\gamma(c)_{1\alpha}\) and 
 \(\iota_{D}\defeq\delta(d)_{2\beta}\) are nondegenerate 
 \Star{}homomorphisms from~\(C\) and~\(D\) to 
 \(\Mult(C\otimes D\otimes\Comp(\Hils[L]))\), respectively. 
 \begin{lemma}
   \label{lemm:Can_action_Drinf_Doub}
   There is a canonical coaction \(\Psi\colon C\boxtimes_{\bichar}D\to C\boxtimes_{\bichar}D\otimes\DrinfDoubAlg_{\bichar}\) 
  of the generalised Drinfeld double~\(\GenDrinfdouble{\G}{\G[H]}{\bichar}\) 
  on~\(C\boxtimes_{\bichar}D\) defined by
  \begin{equation}
   \label{eq:Can_act_def}
    \Psi\iota_{C}(c)=(\iota_{C}\otimes\rho)\gamma(c)
    \quad\text{and}\quad 
    \Psi\iota_{D}(d)=(\iota_{D}\otimes\theta)\delta(d)
  \end{equation}
  for~\(c\in C\), \(d\in D\).
\end{lemma} 
  \begin{proof}
 Define~\(\tilde{\alpha}(a)\defeq(\alpha\otimes\rho)\Comult[A](a)\) 
 and~\(\tilde{\beta}(b)\defeq(\beta\otimes\theta)\Comult[B](b)\) 
 for~\(a\in A\), \(b\in B\). Then~\((\tilde{\alpha},\tilde{\beta})\) is a 
 pair of nondegenerate \Star{}homomorphisms 
 from~\(A\), \(B\) to \(A\otimes B\otimes\Comp(\Hils[L]
 \otimes\Hils[\mathcal{D}]))\).  
 Using~\eqref{eq:W_char_in_second_leg} we get
 \[
      \multunit[A]_{1\tilde{\alpha}}\multunit[B]_{2\tilde{\beta}}
   = \multunit[A]_{1\alpha}\multunit[A]_{1\rho}\multunit[B]_{2\beta}  
      \multunit[B]_{2\theta}
   =  \multunit[A]_{1\alpha}\multunit[B]_{2\beta}\multunit[A]_{1\rho}  
      \multunit[B]_{2\theta}      
   \quad\text{in~\(\U(\hat{A}\otimes\hat{B}\otimes
   \Comp(\Hils[L]\otimes\Hils[\mathcal{D}]))\).}
 \]    
 Now~\((\alpha,\beta)\) satisfies~\eqref{eq:V-Heisenberg_pair} and   
 \((\rho,\theta)\) satisfies~\eqref{eq:V-Drinfeld}; hence we get 
 \[
        \multunit[A]_{1\tilde{\alpha}}\multunit[B]_{2\tilde{\beta}}
  =  \multunit[B]_{2\beta}\multunit[A]_{1\alpha}\bichar_{12}
       \multunit[A]_{1\rho}\multunit[B]_{2\theta}
 =   \multunit[B]_{2\beta}\multunit[B]_{2\theta} 
      \multunit[A]_{1\alpha}\multunit[A]_{1\rho}\bichar_{12}\\
 =\multunit[B]_{2\tilde{\beta}}\multunit[A]_{1\tilde{\alpha}}\bichar_{12}.     
 \]  
 Thus~\((\tilde{\alpha},\tilde{\beta})\) is a~\(\bichar\)\nb-Heisenberg pair. 
 By~\cite{Meyer-Roy-Woronowicz:Twisted_tensor}*{Theorem 4.6}, 
 there is an isomorphism~\(\Psi\) between~\(C\boxtimes_{\bichar}D\) 
 and \(\gamma(C)_{1\tilde{\alpha}}\cdot\delta(D)_{2\tilde{\beta}} 
 \subset\Mult(C\otimes D\otimes\Comp(\Hils[L]\otimes\Hils_{\mathcal{D}}))\) 
 such that~\(\psi\iota_{C}(c)=\gamma(c)_{1\tilde{\alpha}}\) and 
 \(\psi\iota_{D}(d)=\delta(d)_{2\tilde{\beta}}\) for~\(c\in C\) 
 and~\(d\in D\). We compute 
 \[
    \gamma(c)_{1\tilde{\alpha}}
 =\big((\Id_{C}\otimes(\alpha\otimes\rho)\Comult[A])\gamma(c)\big)_{134}
 =\big((\Id_{C}\otimes\alpha)\gamma\otimes\rho)\gamma(c)\big)_{134}
 =(\iota_{C}\otimes\rho)\gamma(c)
 \] 
 for~\(c\in C\), where the second equality uses~\eqref{eq:right_coaction} 
 for~\(\gamma\). A similar computation for~\(\delta\) gives
 \eqref{eq:Can_act_def}.   
 
 A routine computation using Lemma~\ref{lemm:embedding_to_Gen_Drinf} 
 and~\eqref{eq:right_coaction} for~\(\gamma\) and~\(\delta\) 
 yields \((\Id_{C\boxtimes_{\bichar}D}\otimes
 \Comult[\DrinfDoubAlg_{\bichar}])\Psi
 =(\psi\otimes\Id_{\DrinfDoubAlg_{\bichar}})\Psi\).  
  The Podle\'s condition~\eqref{eq:Podles_cond} for~\(\gamma\) gives 
  \[
      (\iota_{C}\otimes\rho)\gamma(C)(1\otimes\rho(A)\cdot\theta(B))
   = (\iota_{C}\otimes\rho)\big((\gamma(C)\cdot (1\otimes A)\big) 
      \theta(B)_{2}
  = \iota_{C}(C)\otimes\rho(A)\cdot\theta(B).    
  \]
 A similar compution shows \((\iota_{D}\otimes\rho)\delta(D)(1\otimes\rho(A)\cdot\theta(B))=\iota_{D}(D)\otimes\rho(A)\cdot\theta(B)\), 
 hence we get the Podle\'s condition~\eqref{eq:Podles_cond} for~\(\Psi\). 
\end{proof} 
\begin{example}
 \label{ex:Non_commtutative_torous}
  Let~\(\G=\G[H]\) be the compact quantum group~\(A=\Cont(\T^{n})\).  
  Then any coaction of~\(\G\) on~\(A\) is the action of the group \(\T^{n}\) 
  by translation. A bicharacter~\(\bichar\in\U(\hat{A}\otimes\hat{A})\) is 
  a map~\(\bichar\colon\Z^{n}\times\Z^{n}\to\T\) which is multiplicative in
  each variable: \(\bichar((a_{n}),(b_{n}))\defeq\Pi_{i,j=1}^{n}
  \lambda_{ij}^{a_{i}\cdot b_{j}}\) for some \((\lambda_{i,j})_{1\leq i,j\leq n}\in\T\).
  The associated~\(\bichar\)\nb-Heisenberg pair~\(((U_{n}),(V_{n}))\) is a pair of \(n\)\nb-tuples  
  of unitaries with the following commutation relations: 
  \(U_{i}U_{j}=U_{j}U_{i}\), \(V_{i}V_{j}=V_{j}V_{i}\), and 
  \(V_{i}U_{j}=U_{j}V_{i}\lambda_{ij}\) for~\(i,j\in\{1,\cdots n\}\). The resulting 
  twisted tensor product~\(A\boxtimes_{\bichar}A\) is the noncommuative \(2n\)\nb-torus. 
  Example~\ref{ex:comm_case_Gen_drinf_doub} shows that the \(\bichar\)\nb-Drinfeld double is 
  \(\Cont(\T^{n}\times\T^{n})\). Thus we get the standard product action of 
  \(\T^{n}\times\T^{n}\) on the noncommutative \(2n\)\nb-torus.
\end{example}
 The coaction \(\Psi\) in Lemma~\ref{lemm:Can_action_Drinf_Doub} 
 generalises the product action of groups. Therefore \(\Psi\) is called the 
 \emph{generalised product of coactions} and denoted by 
 \(\gamma\boxtimes_{\bichar}\delta\).
\begin{proof}[Proof of Theorem~\textup{\ref{the:bifunctor}}]
 By virtue of \cite{Meyer-Roy-Woronowicz:Twisted_tensor}*{Lemma 5.5}, we already 
 know that \(\boxtimes_{\bichar}\) is a bifunctor from \(\Cstcat(\G)\times\Cstcat(\G[H])\) 
 to \(\Cstcat\). More precisely,  given a \(\G\)\nb-equivariant ``map'' \(f\colon (C,\gamma)\to (C_{1},\gamma_{1})\) 
 and a \(\G[H]\)\nb-equivariant ``map'' \(g\colon (D,\delta)\to (D_{1},\delta_{1})\), 
 there is a unique ``map'' \(f\boxtimes_{\bichar}g\colon C\boxtimes_{\bichar}D\to
 C_{1}\boxtimes_{\bichar}D_{1}\) defined by 
    \begin{equation}
   \label{eq:tens_morph}
       (f\boxtimes_{\bichar}g) (\iota_{C}(c))
    = \iota_{C_{1}}(f(c)), 
    \qquad
       (f\boxtimes_{\bichar}g) (\iota_{D}(d))
   = \iota_{D_{1}}(g(d))
 \end{equation}
 for all~\(c\in C\), \(d\in D\). Therefore, we only need to show that 
 the ``map'' \(f\boxtimes_{\bichar}g \colon (C\boxtimes_{\bichar}D,\gamma\boxtimes_{\bichar}\delta)\to 
   (C_{1}\boxtimes_{\bichar}D_{1},\gamma_{1}\boxtimes_{\bichar}\delta_{1})\) 
  is \(\GenDrinfdouble{\G}{\G[H]}{\bichar}\)\nb-equivariant.

  Using~\eqref{eq:tens_morph}, we get
  \begin{equation}
    \label{eq:aux1_equiv}
     (\gamma_{1}\boxtimes_{\bichar}\delta_{1})(f\boxtimes_{\bichar}g)(\iota_{C}(c)\iota_{D}(d))
   = (\gamma_{1}\boxtimes_{\bichar}\delta_{1})(\iota_{C_{1}}(f(c))\iota_{D_{1}}(g(d))).
  \end{equation}
  Now~\eqref{eq:Can_act_def} and the equivariance condition for~\(f\) give
  \[ 
      (\gamma_{1}\boxtimes_{\bichar}\delta_{1})\iota_{C_{1}}(f(c))
   = (\iota_{C_{1}}\otimes\rho)\gamma_{1}(f(c))
   = (\iota_{C_{1}} f\otimes\rho)\gamma(c).
  \]   
  Similarly, we have~\((\gamma_{1}\boxtimes_{\bichar}\delta_{1})\iota_{D_{1}}(g(d))
  =(\iota_{D_{1}} g\otimes\theta)\delta(d)\). Combining the last 
  two equations with \eqref{eq:aux1_equiv} completes the proof:
  \begin{align*}
         (\gamma'\boxtimes_{\bichar}\delta')(f\boxtimes_{\bichar}g)(\iota_{C}(c)\iota_{D}(d))
   &=  (\iota_{C'} f\otimes\rho)\gamma(c) (\iota_{D'} g\otimes\theta)\delta(d)\\
   &= (f\boxtimes_{\bichar}g\otimes\Id_{\DrinfDoubAlg_{\bichar}})
         (\gamma\boxtimes_{\bichar}\delta)(\iota_{C}(c)\iota_{D}(d)).\qedhere
 \end{align*}  
 \end{proof} 
     
\subsection{\texorpdfstring{$\textup{R}$}{R}-matrix on Drinfeld doubles}
 \label{subsec:R_matrix}
    \begin{definition}
     \label{def:R_matrix}
      A bicharacter~\(R\in\U(A\otimes A)\) is called an~\emph{\(\textup{R}\)\nb-matrix} 
     on a quantum group~\(\Qgrp{G}{A}\) if 
    \begin{equation}
     \label{eq:R-matrix}
       R(\flip\Comult[A](a))R^{*}=\Comult[A](a) 
      \qquad\text{for all~\(a\in A\),}
    \end{equation}
    where~\(\flip\) is the standard flip on~\(A\otimes A\).
  \end{definition}     
 Let~\(\hat{B}=A\), \(\DuComult[B]=\Comult[A]\), and 
 \(\bichar=\multunit[A]\in\U(\hat{A}\otimes A)\), and recall the 
 \(\G\)\nb-Drinfeld double~\(\Drinfdouble{\G}\) from 
 Example~\ref{ex:Gen_Drinf_doub_qnt_grp}.
\begin{lemma}
  The unitary~\(R\defeq(\theta\otimes\rho)\multunit[A]\in\U(\DrinfDoubAlg^{A}\otimes
  \DrinfDoubAlg^{A})\) is an~\textup{R}\nb-matrix on the~\(\G\)\nb-Drinfeld double 
  \(\Drinfdouble{\G}\).
\end{lemma} 
 \begin{proof}
  The bicharacter conditions~\eqref{eq:bichar_char_in_first_leg} and~\eqref{eq:bichar_char_in_second_leg} 
  for \(R\) follow from Lemma~\ref{lemm:embedding_to_Gen_Drinf} 
  and~\eqref{eq:Explicit_comult_of_Gen_Drinf_doub}. The comultiplication~\(\Comult[A]\) is the 
  left and right quantum group homomorphism associated 
  to~\(\multunit[A]\in\U(\hat{A}\otimes A)\). Therefore, identifying 
  \(B=\hat{A}\) and \(\Comult[B]=\DuComult[A]\), we 
 rewite~(3) of  Lemma~\ref{lemm:equiv_cond_V_Drinfeld_pair} as 
 \begin{equation}
   \label{eq:R_mat_aux_1}
    (\rho\otimes\rho)\Comult[A](a)=(\multunit[A]_{\theta\rho})(\flip\Comult[A](a))
    (\multunit[A]_{\theta\rho})^{*}\quad\text{for all~\(a\in A\).}
  \end{equation}  
 Similarly, identifying~\(\hat{\Delta}_{L}=\hat{\Delta}_{R}=\DuComult[A]\) and~\(B=\hat{A}\) 
 in~(4) of Lemma~\ref{lemm:equiv_cond_V_Drinfeld_pair} gives
 \begin{equation}
   \label{eq:R_mat_aux_2}
      (\theta\otimes\theta)\DuComult[A](\hat{a})
   = (\multunit[A]_{\theta\rho})(\flip\DuComult[A](\hat{a}))
    (\Dumultunit[A]_{\theta\rho})^{*}\quad\text{for all~\(\hat{a}\in\hat{A}\).}
  \end{equation}
  Combining~\eqref{eq:R_mat_aux_1}, \eqref{eq:R_mat_aux_2}, and
  using~\eqref{eq:Explicit_comult_of_Gen_Drinf_doub}  we obtain~\eqref{def:R_matrix} 
  for~\(R=(\theta\otimes\rho)\multunit[A]\).
\end{proof}  

\section{Properties of generalised quantum codoubles}
 \label{sub:Prop_Gen_qnt_codoub}  
 The definition of a closed quantum subgroup \emph{in the sense of Woronowicz} 
 (see~\cite{Daws-Kasprzak-Skalski-Soltan:Closed_qnt_subgrps}*{Definition 3.2}) 
 uses the notion of a~\(\Cst\)\nb-algebra generated by a quantum family of 
 multipliers. Equivalently, a~\(\Cst\)\nb-quantum group~\(\G[H]_{1}=\Bialg{B_{1}}\) 
 is a closed quantum subgroup of a~\(\Cst\)\nb-quantum group 
 \(\G_{1}=\Bialg{A_{1}}\) if there is a bicharacter  \(\bichar_{1}\in\U(\hat{A_{1}}\otimes B_{1})\) 
 such that the norm closure of \(\{(\omega\otimes\Id_{B_{1}})\bichar' :\omega\in\hat{A_{1}}'\}\) 
 is~\(B_{1}\) (see~\cite{Daws-Kasprzak-Skalski-Soltan:Closed_qnt_subgrps}*{Theorem 3.6 (2)}).
 
\begin{proposition}
 \protect\label{lemm:GenCodouble_closed_sub_grps}
  \(\DuG\) and~\(\DuG[H]\) are closed quantum subgroups of 
  \(\GenCodouble{\G}{\G[H]}{\bichar}\) in the sense of Woronowicz.
 \end{proposition}
 \begin{proof}
 The bicharacter \(\Dumultunit[A]_{\rho2}\in\U(\DrinfDoubAlg_{\bichar}\otimes\hat{A})\)
 corresponds to a quantum group homomorphism from \(\GenCodouble{\G}{\G[H]}{\bichar}\) 
 to~\(\DuG\). Furthermore, the slices~\((\omega\otimes\Id_{\hat{A}})\Dumultunit[A]_{\rho 2}\)  
 for~\(\omega\in(\DrinfDoubAlg^{\bichar})'\) are dense in~\(\hat{A}\). Hence~\(\DuG\) 
 is a closed quantum subgroup of~\(\GenCodouble{\G}{\G[H]}{\bichar}\) in the sense of   
 Woronowicz. Also, \(\Dumultunit[B]_{\theta2}\in\U(\DrinfDoubAlg_{\bichar}\otimes\hat{B})\)
  yields a  similar conclusion for \(\DuG[H]\).
\end{proof}
 \subsection{Coactions and corepresentations} 
  \label{subsec:coact_genodoub}
  \begin{definition}
   \label{def:Gen_Yetter_Drinfeld}
   A~\(\Cst\)\nb-algebra \(C\) along with the coactions~\(\gamma\colon C\to C\otimes\hat{A}\) 
   and~\(\delta\colon C\to C\otimes\hat{B}\) of~\(\DuG\) and~\(\DuG[H]\) 
   is called a \emph{\textup{(}right, right\textup{)} 
  \(\bichar\)\nb-Yetter-Drinfeld} \(\Cst\)\nb-algebra if the following diagram commutes:
       \begin{equation}
      \label{eq:Gen_Yetter_Drinfeld}
        \begin{tikzpicture}[baseline=(current bounding box.west)]
         \matrix(m)[cd,column sep=6em]{
           C              & C\otimes \hat{B}             & C\otimes\hat{A}\otimes\hat{B}\\
           C\otimes\hat{A}& C\otimes\hat{B}\otimes\hat{A}& C\otimes\hat{A}\otimes\hat{B}\\
           };
         \draw[cdar] (m-1-1) -- node {\(\delta\)}                  (m-1-2);
         \draw[cdar] (m-1-2) -- node {\(\gamma\otimes\Id_{\hat{B}}\)} (m-1-3);
         \draw[cdar] (m-2-3) -- node[swap] {\(\Id_{C}\otimes \textup{Ad}_{\bichar}\)}                (m-1-3);
         \draw[cdar] (m-1-1) -- node[swap] {\(\gamma\)} (m-2-1);
         \draw[cdar] (m-2-1) -- node {\(\delta\otimes\Id_{\hat{A}}\)} (m-2-2);
         \draw[cdar] (m-2-2) -- node {\(\Id_{C}\otimes\flip\)} (m-2-3);
       \end{tikzpicture}
       \quad .
    \end{equation}
 Let~\(\mathcal{YD}(\DuG,\DuG[H],\bichar)\) be the category with 
 \(\bichar\)\nb-Yetter-Drinfeld \(\Cst\)\nb-algebras as objects and 
 \(\DuG\)\nb- and~\(\DuG[H]\)\nb-equivariant morphism as arrows.    
 \end{definition} 
 \begin{example}
  Consider~\(\hat{A}=\Contvin(G)\) and~\(\hat{B}=\Contvin(H)\) for locally compact
  groups~\(G\) and~\(H\). Then any \(G\)\nb-\(\Cst\)\nb-algebra with trivial~\(H\)\nb-coaction 
  makes it Yetter-Drinfeld in this generalised sense.
 \end{example}
 \begin{example}  
  In particular, let~\(\hat{B}=A\), \(\DuComult[B]=\Comult[A]\),  
  and~\(\bichar=\multunit[A]\in\U(\hat{A}\otimes A)\). 
  Then \(\multunit[A]\)\nb-Yetter-Drinfeld \(\Cst\)\nb-algebras 
  are the same as \(\G\)\nb-Yetter-Drinfeld \(\Cst\)\nb-algebras defined 
  by Nest and Voigt in~\cite{Nest-Voigt:Poincare}. The category of 
  \(\G\)\nb-Yetter-Drinfeld \(\Cst\)\nb-algebras is denoted by 
  \(\mathcal{YD}(\G)\).
\end{example} 
 Proposition~\(3.2\) from~\cite{Nest-Voigt:Poincare} shows 
 that the categories \(\mathcal{YD}(G)\) and~\(\Cstcat(\Codouble{\G})\) are 
 equivalent for a regular~\(\Cst\)\nb-quantum group \(\G\) 
 with Haar weights (because it uses the \(\Cst\)\nb-algebraic 
 picture from~\cite{Baaj-Vaes:Double_cros_prod}). We generalise 
 this fact in the next proposition:
 \begin{proposition}
  \label{prop:V-Yetter-Drinfeld_Gen_cod}
   Every~\(\GenCodouble{\G}{\G[H]}{\bichar}\)\nb-\(\Cst\)\nb-algebra 
   is a \(\bichar\)\nb-Yetter-Drinfeld \(\Cst\)\nb-algebra, and vice versa.
 \end{proposition}
 
  Define~\(\Delta_R\colon\CodoubAlg_{\bichar}\to\CodoubAlg_{\bichar}\otimes\hat{A}\) 
  by~\(\Delta_R\defeq(\Id_{\hat{B}}\otimes\DuComult[A])\). 
  Equation~\eqref{eq:def_V_via_right_homomorphism} and~\eqref{eq:W_char_in_second_leg} 
  for~\(\Dumultunit[A]\) give
 \begin{equation*}
    (\Id_{\DrinfDoubAlg_{\bichar}}\otimes\Delta_R)\DuMultunit[\mathcal{D}]
   = (\Id_{\DrinfDoubAlg_{\bichar}}\otimes\Delta_R)\big(\Dumultunit[B]_{\theta2}
      \Dumultunit[A]_{\rho3}\big)
  =\Dumultunit[B]_{\theta2}\Dumultunit[A]_{\rho3}\Dumultunit[A]_{\rho4}
 \end{equation*} 
 in~\(\U(\DrinfDoubAlg_{\bichar}\otimes\hat{B}\otimes\hat{A}\otimes\hat{A})\). 
 Collapsing the second and third leg in the last computation we obtain 
 \begin{equation}
    \label{eq:proj_1}
      (\Id_{\DrinfDoubAlg_{\bichar}}\otimes\Delta_R)\DuMultunit[\mathcal{D}]
   = \DuMultunit[\mathcal{D}]_{12}\Dumultunit[A]_{\rho 3}
   \qquad\text{in~\(\U(\DrinfDoubAlg_{\bichar}\otimes\CodoubAlg_{\bichar}\otimes\hat{A})\)}.
 \end{equation}
 Comparing the last expression with~\eqref{eq:def_V_via_right_homomorphism}, 
 we conclude  that~\(\Delta_R\) is the right quantum group homomorphism 
 corresponding to the bicharacter 
 \(\Dumultunit[A]_{\rho2}\in\U(\DrinfDoubAlg_{\bichar}\otimes\hat{A})\). 
 
 Similarly, using~\eqref{eq:Exist_Gen_codoub_compatible_corep} we can show that 
 \(\Delta'_R\colon\CodoubAlg_{\bichar}\to\CodoubAlg_{\bichar} 
  \otimes\hat{B}\) defined by~\(\Delta'_R(\hat{b}\otimes\hat{a})
  \defeq\flip_{23}^{\bichar}(\DuComult[B](\hat{b})\otimes\hat{a})\) 
  satisfies 
  \begin{equation}
   \label{eq:proj_2}
   (\Id_{\DrinfDoubAlg_{\bichar}}\otimes\Delta'_R)\DuMultunit[\mathcal{D}]
    = \DuMultunit[\mathcal{D}]_{12}\Dumultunit[B]_{\theta 3} 
   \qquad\text{in~\(\U(\DrinfDoubAlg_{\bichar}\otimes\CodoubAlg_{\bichar}\otimes\hat{B})\).}
   \end{equation} 
  Hence \(\Delta'_{R}\) is the right quantum group homomorphism associated to the bicharacter 
  \(\Dumultunit[B]_{\theta2}\in\U(\DrinfDoubAlg_{\bichar}\otimes\hat{B})\).
\begin{lemma}
  \label{prop:Gen_YD_exist}
   \(\CodoubAlg_{\bichar}\) is a~\(\bichar\)\nb-Yetter-Drinfeld algebra.
\end{lemma}   
 \begin{proof}
 Using~\eqref{eq:Exist_Gen_codoub_compatible_corep} we compute
  \begin{align*}
    \flip^{\bichar}_{34}
    \Big(\big(\Id_{\DrinfDoubAlg_{\bichar}}\otimes((\Delta'_{R}\otimes\Id_{\hat{A}})\Delta_{R})\big)\Dumultunit[B]_{\theta2}
    \Dumultunit[A]_{\rho3}\Big)
   &= \Dumultunit[B]_{\theta2}\flip^{\bichar}_{34}\big(\Dumultunit[A]_{\rho3}\Dumultunit[B]_{\theta4}\big)\Dumultunit[A]_{\rho5}\\
   &= \Dumultunit[B]_{\theta2}\Dumultunit[B]_{\theta3}\Dumultunit[A]_{\rho4}\Dumultunit[A]_{\rho5}\\
   &= \big(\Id_{\DrinfDoubAlg_{\bichar}}\otimes((\Delta_{R}\otimes\Id_{\hat{A}})\Delta'_{R})\big)
       \Dumultunit[B]_{\theta2}\Dumultunit[A]_{\rho3}.
  \end{align*}
  Taking slices on the first leg by functionals on~\(\DrinfDoubAlg_{\bichar}\) shows that 
  \(\CodoubAlg_{\bichar}\) is a~\(\bichar\)\nb-Yetter-Drinfeld \(\Cst\)\nb-algebra 
  with respect to the coactions \(\Delta_{R}\) and \(\Delta'_{R}\) of~\(\DuG\) and~\(\DuG[H]\) 
  on~\(\CodoubAlg_{\bichar}\).
\end{proof} 

  \begin{proof}[Proof of Proposition~\textup{\ref{prop:V-Yetter-Drinfeld_Gen_cod}}]  
   Let~\(C\) be a~\(\GenCodouble{\G}{\G[H]}{\bichar}\)\nb-\(\Cst\)\nb-algebra.
   Now \cite{Meyer-Roy-Woronowicz:Twisted_tensor}*{Lemma 2.9} identifies 
   \(C\) with a subalgebra of~\(\Mult(C'\otimes\CodoubAlg_{\bichar})\) 
   for some \(\Cst\)\nb-algebra~\(C'\) with the coaction of \(\GenCodouble{\G}{\G[H]}{\bichar}\)
   only on~\(\CodoubAlg_{\bichar}\). By 
   Proposition~\ref{prop:Gen_YD_exist}, \(\CodoubAlg_{\bichar}\) 
   is a~\(\bichar\)\nb-Yetter-Drinfeld \(\Cst\)\nb-algebra, hence so is \(C\).
   
 Conversely, let~\(\gamma\colon C\to C\otimes\hat{A}\) and 
 \(\delta\colon C \to C\otimes\hat{B}\) satisfy~\eqref{eq:Gen_Yetter_Drinfeld}. Define a nondegenerate, 
  injective \(^{*}\)\nb-homomorphism~\(\tilde{\gamma}\colon C\to C\otimes\CodoubAlg^{\bichar}\) by 
  \(\tilde{\gamma}\defeq (\delta\otimes\Id_{\hat{A}})\gamma\). 
  
  The Podle\'s condition~\eqref{eq:Podles_cond} for~\(\tilde{\gamma}\) is induced 
  from those for~\(\gamma\) and \(\delta\) in the following way: 
  \begin{align*}
   \tilde{\gamma}(C)\cdot (1_{C}\otimes\CodoubAlg^{\bichar})
   &=\big((\delta\otimes\Id_{\hat{A}})(\gamma(C)\cdot (1_{C}\otimes\hat{A})\big)
         \cdot(1_{C}\otimes\hat{B}\otimes 1_{\hat{A}})\\
   &= \big(\delta(C)\cdot (1_{C}\otimes\hat{B})\big)\otimes\hat{A}
   =C\otimes\CodoubAlg_{\bichar}.
 \end{align*}

 The following computation yields~\eqref{eq:right_coaction} for~\(\tilde{\gamma}\):
 \begin{align*}
     (\tilde{\gamma}\otimes\Id_{\hat{B}\otimes\hat{A}})\tilde{\gamma}  
   &= (\delta\otimes\Id_{\hat{A}\otimes\hat{B}\otimes\hat{A}})
     ((\gamma\otimes\Id_{\hat{B}})\delta\otimes\Id_{\hat{A}})\gamma\\
  &= \flip_{34}^{\bichar}\big((\delta\otimes\Id_{\hat{B}})\delta\otimes\Id_{\hat{A}\otimes\hat{A}})
     (\gamma\otimes\Id_{\hat{A}})\gamma\big)\\
  &= \flip_{34}^{\bichar}\big((\Id_{C}\otimes\DuComult[B]\otimes\DuComult[A])(\delta\otimes\Id_{\hat{A}})\gamma
     \big)\\
  &= (\Id_{C}\otimes\DuComult[\DrinfDoubAlg_{\bichar}])\tilde{\gamma}.
 \end{align*}
 The first equality is trivial, the second equality uses~\eqref{eq:Gen_Yetter_Drinfeld}, 
 the third equality uses~\eqref{eq:corep_cond}, and the last equality 
 uses~\eqref{eq:Gen_Co_doub_def}.
\end{proof}

 Let \(\corep{U}^{\DuG}\in\U(\Comp(\Hils[K])\otimes\hat{A})\) and 
 \(\corep{U}^{\DuG[H]}\in\U(\Comp(\Hils[K])\otimes\hat{B})\) be 
 corepresentations of \(\DuG\) and~\(\DuG[H]\) on~\(\Hils[K]\).
 \begin{definition} 
    \label{def:codouble_compatible_corep}
    A pair~\((\corep{U}^{\DuG},\corep{U}^{\DuG[H]})\) is 
    called~\emph{\(\GenCodouble{\G}{\G[H]}{\bichar}\)\nb-compatible} 
   if \(\corep{U}^{\DuG}\) and~\(\corep{U}^{\DuG[H]}\) commute in the following way:
  \begin{equation}
   \label{eq:comm_reln_compatible_corep}
     \flip_{23}^{\bichar}(\corep{U}^{\DuG[H]}_{12}\corep{U}^{\DuG}_{13})
   = \corep{U}^{\DuG[H]}_{13}\corep{U}^{\DuG}_{12}
   \qquad\text{in~\(\U(\Comp(\Hils[K])\otimes\hat{A}\otimes\hat{B})\).}
  \end{equation}
\end{definition}
\begin{example} 
  Equation~\eqref{eq:rho_theta_equiv_Drinfeld_cond} shows that 
  the pair \((\Dumultunit[A]_{\rho 2},\Dumultunit[B]_{\theta 2})\) of 
  corepresentations of~\(\DuG\) and~\(\DuG[H]\) on~\(\Hils_{\mathcal{D}}\) 
 is \(\GenCodouble{\G}{\G[H]}{\bichar}\)\nb-compatible. 
 This is the corepresentation version of Lemma~\ref{prop:Gen_YD_exist}.
\end{example} 
 \begin{proposition}
  \label{prop:corep_codoub_classify}
   Corepresentations of~\(\GenCodouble{\G}{\G[H]}{\bichar}\) are 
   in one\nb-to\nb-one correspondence with 
   \(\GenCodouble{\G}{\G[H]}{\bichar}\)\nb-compatible 
   pairs of corepresentations.
  \end{proposition} 
  \begin{proof}
   A routine computation shows that any 
   \(\GenCodouble{\G}{\G[H]}{\bichar}\)\nb-compatible 
   pair of corepresentations \((\corep{U}^{\DuG},\corep{U}^{\DuG[H]})\)
   on~\(\Hils[K]\) gives a corepresentation 
   \(\corep{X}\in\U(\Comp(\Hils[K])\otimes\CodoubAlg_{\bichar})\) 
   of~\(\GenCodouble{\G}{\G[H]}{\bichar}\) by  
   \begin{equation}
    \label{eq:codob_corep_frm_corep_qgrp_and_dual}
     \corep{X}\defeq\corep{U}^{\DuG[H]}_{12}\corep{U}^{\DuG}_{13} 
     \qquad\text{ in~\(\U(\Comp(\Hils[K])\otimes\hat{B}\otimes\hat{A})\).}
   \end{equation}
  Conversely, let \(\corep{X}\in\U(\Comp(\Hils[K])\otimes\CodoubAlg_{\bichar})\) 
   be a corepresentation of~\(\GenCodouble{\G}{\G[H]}{\bichar}\) on~\(\Hils[K]\).
   By~\cite{Meyer-Roy-Woronowicz:Homomorphisms}*{Proposition 6.5} or 
   \cite{Roy:Qgrp_with_proj}*{Proposition 3.31} the right quantum group 
   homomorphism~\(\Delta_{R}\) in~\eqref{eq:proj_1} induces a 
   corepresentation~\(\corep{U}\in\U(\Comp(\Hils[K])\otimes\hat{A})\) 
   of~\(\DuG\) on~\(\Hils[K]\) such that 
   \[
     (\Id_{\Hils[K]}\otimes\Delta_{R})\corep{X}=\corep{X}_{12}\corep{U}^{\DuG}_{13} 
     \qquad\text{in~\(\U(\Comp(\Hils[K])\otimes\CodoubAlg_{\bichar}\otimes\hat{A})\).}
   \]     
   Similarly, the right quantum group homomorphism~\(\Delta'_{R}\) in 
   \eqref{eq:proj_2} gives a corepresentation \(\corep{U}^{\DuG[H]}\in\U(\Comp(\Hils[K])\otimes 
   \hat{B})\) of~\(\DuG[H]\) satisfying 
   \[
     (\Id_{\Hils[K]}\otimes\Delta'_{R})\corep{X}=\corep{X}_{12}\corep{U}^{\DuG[H]}_{13} 
     \qquad\text{in~\(\U(\Comp(\Hils[K])\otimes\CodoubAlg_{\bichar}\otimes\hat{B})\).}
   \]      
   Lemma~\ref{prop:Gen_YD_exist} gives
   \begin{align*}
       \corep{X}_{12}\corep{U}^{\DuG}_{13}\corep{U}^{\DuG[H]}_{14} 
     = \big(\Id_{\Hils[K]}\otimes(\Delta_{R}\otimes\Id_{\hat{A}})\Delta'_{R}\big)\corep{X}
     &= \flip_{34}^{\bichar}\big(\Id_{\Hils[K]}\otimes(\Delta'_{R}\otimes\Id_{\hat{A}})\Delta_{R}\big)\corep{X}\\
     &=\flip_{34}^{\bichar}(\corep{X}_{12}\corep{U}^{\DuG[H]}_{13}\corep{U}^{\DuG}_{14})
       =\corep{X}_{12}\flip^{\bichar}_{34}(\corep{U}^{\DuG[H]}_{13}\corep{U}^{\DuG}_{14}).\qedhere
  \end{align*} 
 \end{proof}

\section*{Acknowledgements}
  This work was supported by the 
  Deutsche Forschungsgemeinschaft (DFG) through the Research
  Training Group 1493 and the Institutional Strategy of the 
  University of G\"ottingen, during the doctoral studies of the author. 
  Parts of the manuscript had been revised while the author had been 
  supported by Fields--Ontario postdoctoral fellowship, NSERC and ERA at the 
  Fields Institute, and University of Ottawa. The author gratefully thanks Professors 
  Ralf Meyer and Stanis\l{}aw Lech Woronowicz for their helpful 
  discussions and comments. In addition author would 
  like to thank Karen Strung and anonymous referee 
  for an exceptionally careful reading of the article and 
  valuable comments that led to substantial improvements 
  of its content.

\end{document}